\documentclass[11pt]{article}

\usepackage{amsthm}
\usepackage{amsfonts}
\usepackage{amsmath}
\usepackage{amssymb}
\usepackage{color}
\usepackage{verbatim}
\usepackage{pst-all}
\usepackage{caption2}
\usepackage[dvips]{graphics}

\newtheorem{theorem}{Theorem}[section] 
\newtheorem{lemma}[theorem]{Lemma}     

\newtheorem{definition}[theorem]{Definition}

\newcommand{\display}{\begin{displaymath}}
\newcommand{\offplay}{\end{displaymath}}
\newcommand{\eqn}{\begin{eqnarray}}
\newcommand{\eqns}{\begin{eqnarray*}}
\newcommand{\uneq}{\end{eqnarray}}
\newcommand{\uneqs}{\end{eqnarray*}}

\newcommand{\nn}{\mathbb{N}}

\newcommand{\qq}{\mathbb{Q}}
\newcommand{\mcal}{\mathcal}
\newcommand{\m}{\mathcal}

\newcommand{\back}{\backslash}

\newcommand{\nc}{{$2822$ }}
\newcommand{\ncmath}{2822}

\newcommand{\cR}{\texttt{r}}
\newcommand{\cG}{\texttt{g}}
\newcommand{\cB}{\texttt{b}}
\newcommand{\cY}{\texttt{y}}

\newcommand{\rA}[1]{\texttt{P}_{{#1}}}
\newcommand{\rB}[1]{\texttt{Q}_{#1}}
\newcommand{\rC}{\texttt{C}}

\newcommand{\gA}{A}
\newcommand{\gB}{B}
\newcommand{\gH}{A}

\newcommand{\rX}{\texttt{x}}
\newcommand{\rY}{\texttt{y}}
\newcommand{\rZ}{\texttt{z}}
\newcommand{\rW}{\texttt{w}}

\newcommand{\gm}{\gamma}

\newcommand{\url}{{\tt http://www.math.ubc.ca/$\sim$jpsteinb}}
\newcommand{\arxurl}{{\tt www.arxiv.org}}

\newif\ifarxiv

\arxivtrue

\newcommand{\height}{0}

\newcommand{\usualfigureheader}{
\begin{figure}
\begin{center}
\begin{pspicture}(0,0)(0,\height)
}

\newcommand{\exclamatoryusualfigureheader}{
\begin{figure}[!!h]
\begin{center}
\begin{pspicture}(0, 0)(0, \height)
}

\input{presentation_text}
\input{deg_12_proof_macros}

\newcommand{\presentationfig}{
\renewcommand{\height}{17.5}
\usualfigureheader
\rput(-0.3,0){
\rput(-6.6,-10.3){
\includegraphics{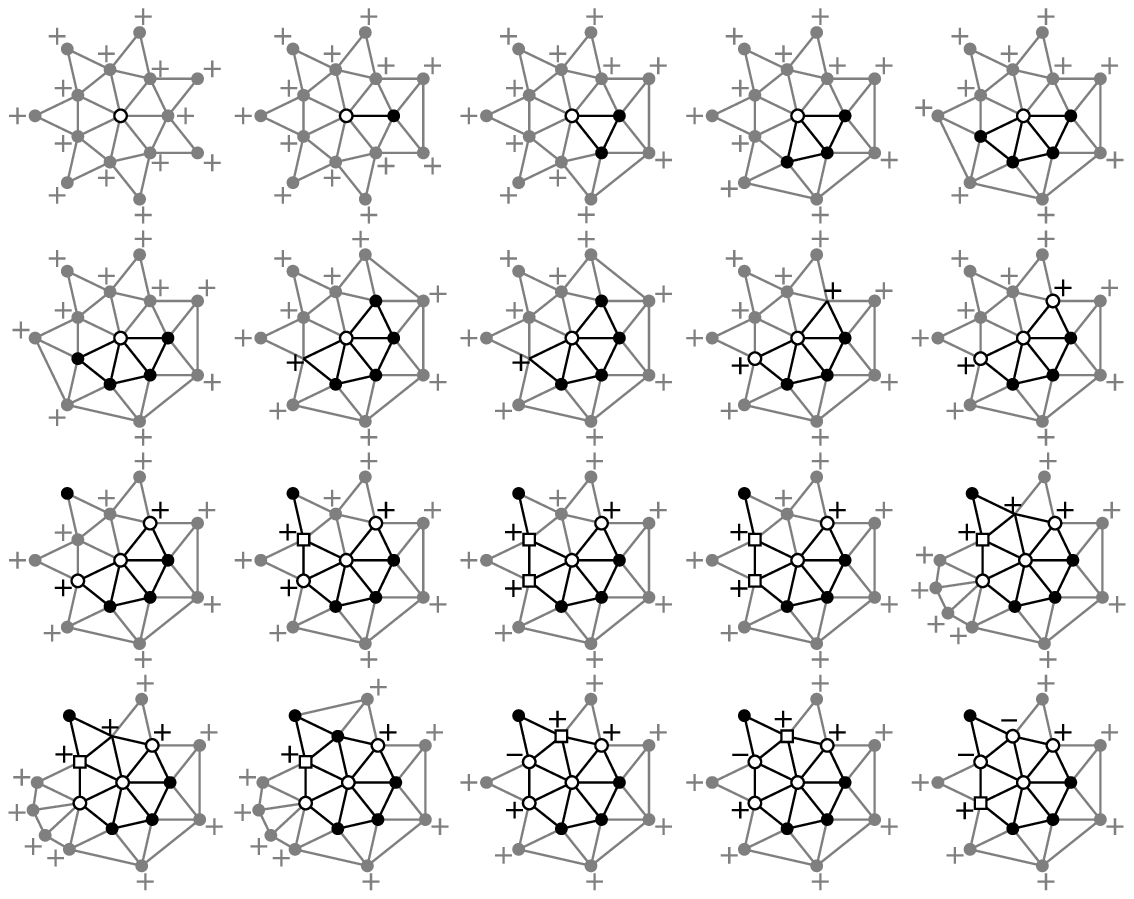}
\rput(1.2,27){
\rput[r](0,0){{\footnotesize $\mathit{0}$}}
\rput[r](0,-2.3){{\footnotesize $\mathit{5}$}}
\rput[r](0,-4.5){{\footnotesize $\mathit{10}$}}
\rput[r](0,-6.8){{\footnotesize $\mathit{15}$}}
}
}
\rput(-5.2, 7.5){\ptext}
}
\end{pspicture}
\renewcommand{\captionlabeldelim}{:}\caption{Illustration of a sample presentation file. Vertices $v$ with $(\gamma^-(v), \gamma^+(v)) = (5, \infty)$ are lightened for clarity.}\renewcommand{\captionlabeldelim}{}
\end{center}
\label{presentation}
\end{figure}
}

\newcommand{\fourintersectionfig}{
\renewcommand{\height}{3.6}
\usualfigureheader
\rput(-2,0.2){\includegraphics{four_intersection.eps}}
\end{pspicture}
\renewcommand{\captionlabeldelim}{:}\caption{How to make a map cubic.}\renewcommand{\captionlabeldelim}{}
\end{center}
\end{figure}
}

\newcommand{\mapandgraphfig}{
\renewcommand{\height}{3.6}
\usualfigureheader
\rput(0.5,0){\includegraphics{map_and_dual.eps}}
\end{pspicture}
\renewcommand{\captionlabeldelim}{:}\caption{Forming the dual graph of a map.}\renewcommand{\captionlabeldelim}{}
\end{center}
\end{figure}
}

\newcommand{\birkhoffdiamondfig}{
\renewcommand{\height}{2.7}
\usualfigureheader
\rput(0.55,-0.4){\includegraphics{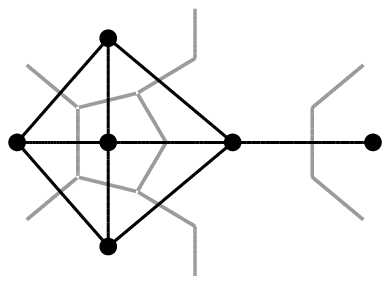}}
\end{pspicture}
\renewcommand{\captionlabeldelim}{:}\caption{The Birkhoff diamond $\gA$ (left) and its reducer $\gA'$ (right), shown with their dual maps.}\renewcommand{\captionlabeldelim}{}
\label{birkhoffdiamond}
\end{center}
\end{figure}
}

\newcommand{\plus}{
\psline(-0.078, 0)(0.078, 0)
\psline(0, -0.078)(0, 0.078)
}

\newcommand{\degreetwelvefig}{
\renewcommand{\height}{13.4}
\usualfigureheader
\psset{linewidth=0.6pt}
\rput(2,-1){
\degtwelveproof
}
\end{pspicture}
\renewcommand{\captionlabeldelim}{:}\caption{The case analysis showing that a vertex of degree 7 cannot transfer more than $1/2$ units of charge to a vertex of degree $\geq 8$.}\renewcommand{\captionlabeldelim}{}
\label{degreetwelve}
\end{center}
\end{figure}
}

\newcommand{\vertexshapesfig}{
\renewcommand{\height}{1.7}
\usualfigureheader
\psset{linewidth=0.6pt}
\rput(-0.1,1.5){
\rput(-5,0){
\rput(0,0)   {\psline(0.19,-0.34)(0,0)(-0.19,-0.34)\psdot[dotsize=0.19]\rput(1.4,-0.08){degree 5}
\rput(0,-0.6){\psline(0.19,-0.34)(0,0)(-0.19,-0.34)\rput(1.4,-0.08){degree 6}
\rput(0,-0.6){\psline(0.19,-0.34)(0,0)(-0.19,-0.34)\pscircle[fillstyle=solid,fillcolor=white]{0.10}\rput(1.4,-0.08){degree 7}
}}}}
\rput(-1,0){
\rput(0,0)   {\psline(0.19,-0.34)(0,0)(-0.19,-0.34)\pspolygon[fillstyle=solid,fillcolor=white](-0.078,-0.078)(-0.078,0.078)(0.078,0.078)(0.078,-0.078)\rput(1.4,-0.08){degree 8}
\rput(0,-0.6){\psline(0.19,-0.34)(0,0)(-0.19,-0.34)\pspolygon[fillstyle=solid,fillcolor=white](-0.1,-0.07)(0,0.1)(0.1,-0.07)\rput(1.4,-0.08){degree 9}
\rput(0,-0.6){\psline(0.19,-0.34)(0,0)(-0.19,-0.34)\rput(0,-0.02){\pspolygon[fillstyle=solid,fillcolor=white](-0.061,-0.04)(-0.1,0.062)(0,0.14)(0.1,0.062)(0.061,-0.04)}\rput(1.5,-0.08){degree 10}
}}}}
\rput(2.8,0){
\rput(0,0)   {
\rput(0,-0.6){\psline(0.19,-0.34)(0,0)(-0.19,-0.34)\pspolygon[fillstyle=solid,fillcolor=white](-0.09,-0.09)(-0.09,0.09)(-0.03,0.09)(-0.03,-0.09)\rput(1.4,-0.08){degree 11}
                                                   \pspolygon[fillstyle=solid,fillcolor=white]( 0.09,-0.09)( 0.09,0.09)( 0.03,0.09)( 0.03,-0.09)
\rput(0,-0.6){
}}}}
}
\end{pspicture}
\renewcommand{\captionlabeldelim}{:}\caption{Shapes for showing the degree of a vertex.}\renewcommand{\captionlabeldelim}{}
\label{vertexshapes}
\end{center}
\end{figure}
}

\newcommand{\sqr}{0.866}
\newcommand{\sqrpp}{1.866}

\newcommand{\fivefivepairfig}{
\renewcommand{\height}{1.5}
\usualfigureheader
\psset{linewidth=0.8pt,xunit=0.69cm,yunit=0.69cm}
\rput(-6.2,1.3){
\pspolygon(-\sqr,0.5)(\sqr,-0.5)(\sqr,0.5)(-\sqr,-0.5)
\psdot[dotsize=0.19](-\sqr,-0.5)
\psdot[dotsize=0.19](-\sqr, 0.5)
\rput(1.732,0){
\psline(-\sqr,-0.5)(0,-1)(0,0)(-\sqr,0.5)
\psline(0,0)(-\sqr,-0.5)
\rput(-\sqr,-0.5){\pscircle[fillstyle=solid,fillcolor=white]{0.113}}
\psdot[dotsize=0.19](0,-1)
}
\rput(5.7,-.2){
\pspolygon(-1,0)(-0.5,-\sqr)(0.5,\sqr)(1.5,-\sqr)(2, 0)
\psdot[dotsize=0.19](0.5,\sqr)
\psline(-1,0)(-1.5,-\sqr)(-0.5,-\sqr)
\psline(2,0)(2.5,-\sqr)(1.5,-\sqr)
\psdot[dotsize=0.19](-1,0)
\psdot[dotsize=0.19](-1.5,-\sqr)
\psdot[dotsize=0.19](-0.5,-\sqr)
\psdot[dotsize=0.19](2,0)
\psdot[dotsize=0.19](2.5,-\sqr)
\psdot[dotsize=0.19](1.5,-\sqr)
\rput(5.7,0){
\psline(0,0)(-\sqr,0.5)(-\sqr,-0.5)(0,0)(1,0)(\sqrpp,0.5)(\sqrpp,-0.5)(1,0)
\psdot[dotsize=0.19](0,0)
\psdot[dotsize=0.19](1,0)
\psdot[dotsize=0.19](-\sqr,0.5)
\psdot[dotsize=0.19](-\sqr,-0.5)
\psdot[dotsize=0.19](\sqrpp,0.5)
\psdot[dotsize=0.19](\sqrpp,-0.5)
}
}
}
\end{pspicture}
\renewcommand{\captionlabeldelim}{:}\caption{Configurations with blocks of size 2 and 3. The first and third configurations have hanging 5-5 pairs.}\renewcommand{\captionlabeldelim}{}
\label{fivefivepair}
\end{center}
\end{figure}
}

\newcommand{\curvesfig}{
\renewcommand{\height}{1.8}
\usualfigureheader
\psset{xunit=0.7cm, yunit=0.7cm}
\rput(-4.3,0){\includegraphics{test.eps}}
\end{pspicture}
\renewcommand{\captionlabeldelim}{:}\caption{A potentially infinite sequence of $D$-reducible configurations, drawn in map form; terms of the sequence have been checked $D$-reducible up to ring-size 18 (the first configuration has ring-size 9).}\renewcommand{\captionlabeldelim}{}
\end{center}
\end{figure}
}

\newcommand{\mayerfig}{
\renewcommand{\height}{5.1}
\usualfigureheader
\rput(-3.4,0){\includegraphics{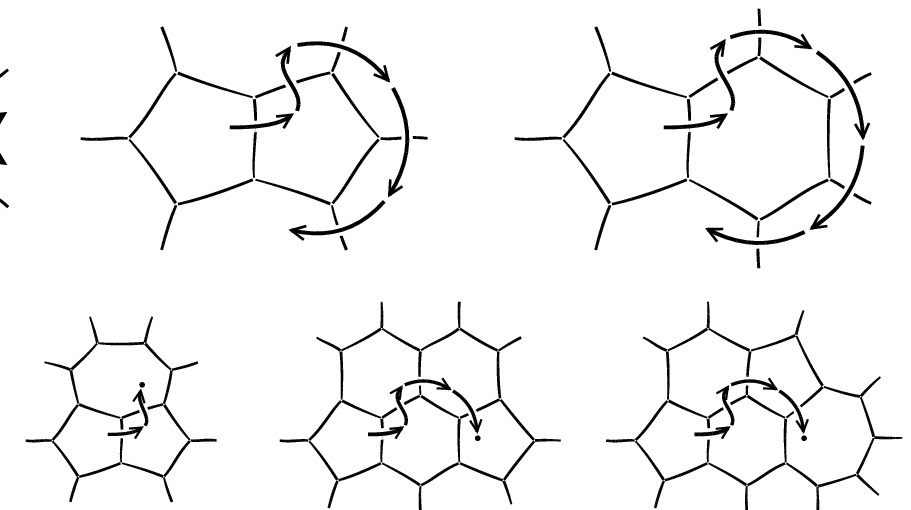}
  \rput(-.15,3.78){
  \rput(-0.7,0){{\large{\textbf{1}}}}
  \rput(0.43,0.32){$\frac{1}{10}$}
  \rput(0.43,-0.3){$\frac{1}{10}$}
  \rput(-0.05,-1.02){$\frac{1}{10}$}
  \rput(-0.75,-1.28){$\frac{1}{10}$}
  \rput(-1.58,-0.95){$\frac{1}{10}$}
  \rput(-1.95,-0.42){$\frac{1}{10}$}
  \rput(-1.95, 0.42){$\frac{1}{10}$}
  \rput(-1.58, 1.02){$\frac{1}{10}$}
  \rput(-0.75, 1.32){$\frac{1}{10}$}
  \rput(-0.05, 1.1){$\frac{1}{10}$}
  }
}
\end{pspicture}
\renewcommand{\captionlabeldelim}{:}\caption{A sketch of Mayer's discharging procedure, using map form. Top: a pentagon divides its charge of 1 into 10 packets and sends two packets to each neighbor. A packet settles if the neighbor is a major vertex, otherwise travels clockwise or counterclockwise around the neighbor until it reaches a major vertex or until a reducible configuration appears. Bottom: sample packet trajectories. In the third drawing, the charge stops because of the appearance of a reducible configuration (the fifth configuration of Fig$.$ \ref{minorconfs}). In the fourth drawing, the charge would stop before the heptagon if Bernhart's diamond (Fig$.$ \ref{bernhart}) were an allowed configuration. If no major vertex appears one of the configurations of Fig$.$ \ref{minorconfs} necessarily appears.}\renewcommand{\captionlabeldelim}{}
\label{mayerrules}
\end{center}
\end{figure}
}

\newcommand{\minorconfsfig}{
\renewcommand{\height}{2.4}
\usualfigureheader
\rput(-7.35,0.5){\includegraphics{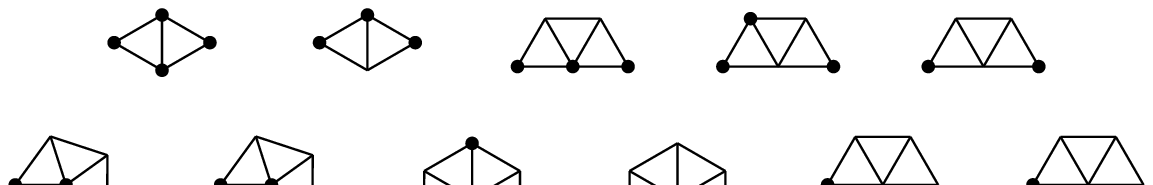}}
\end{pspicture}
\renewcommand{\captionlabeldelim}{:}\caption{The $D$-reducible configurations used for Mayer's discharging procedure.}\renewcommand{\captionlabeldelim}{}
\label{minorconfs}
\end{center}
\end{figure}
}

\newcommand{\proofconfsfig}{
\renewcommand{\height}{3.4}
\usualfigureheader
\rput(-6.8,1.25){\includegraphics{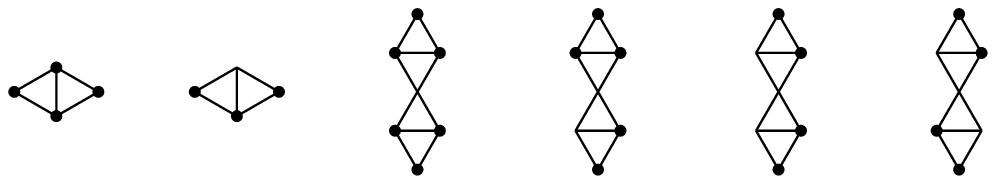}}
\end{pspicture}
\renewcommand{\captionlabeldelim}{:}\caption{The $D$-reducible configurations used for the proof of Lemma \ref{lastly}.}\renewcommand{\captionlabeldelim}{}
\label{proofconfs}
\end{center}
\end{figure}
}

\newcommand{\bernhartfig}{
\renewcommand{\height}{1.0}
\usualfigureheader
\psset{xunit=0.9cm,yunit=0.9cm}
\rput(0,0.4){
\pspolygon[linewidth=0.7pt](0,0.5)(0.866,0)(0,-0.5)(-0.866,0)
\psline(0,0.5)(0,-0.5)
\psdot[dotsize=0.2](0.866,0)
\psdot[dotsize=0.2](-0.866,0)
}
\end{pspicture}
\renewcommand{\captionlabeldelim}{:}\caption{The Bernhart diamond, a $C$-reducible configuration.}\renewcommand{\captionlabeldelim}{}
\label{bernhart}
\end{center}
\end{figure}
}

\newcommand{\rsstrulesfig}{
\renewcommand{\height}{13.6}
\usualfigureheader
\psset{xunit=0.7cm, yunit=0.7cm}
\rput(-9.7,-16.7){
\rput(0.5,-1.5){\includegraphics{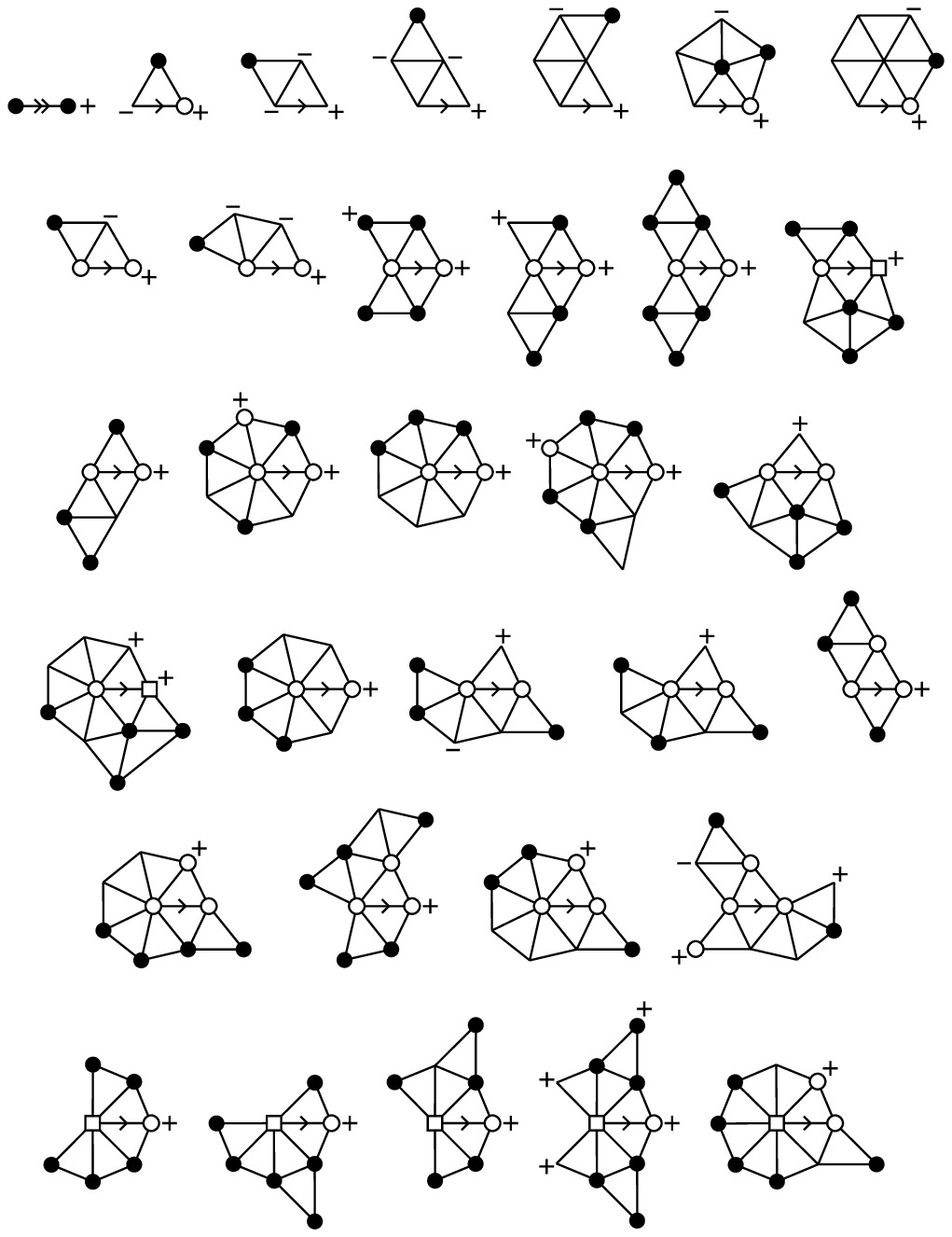}}}
\end{pspicture}
\renewcommand{\captionlabeldelim}{:}\caption{The discharging rules of Roberston et al.}\renewcommand{\captionlabeldelim}{}
\label{rsstrules}
\end{center}
\end{figure}
}

\newcommand{\mainrulesfig}{
\renewcommand{\height}{16.1}
\usualfigureheader
\psset{xunit=0.7cm, yunit=0.7cm}
\rput(-10.4,-13){
\rput(0.7,-1.5){\includegraphics{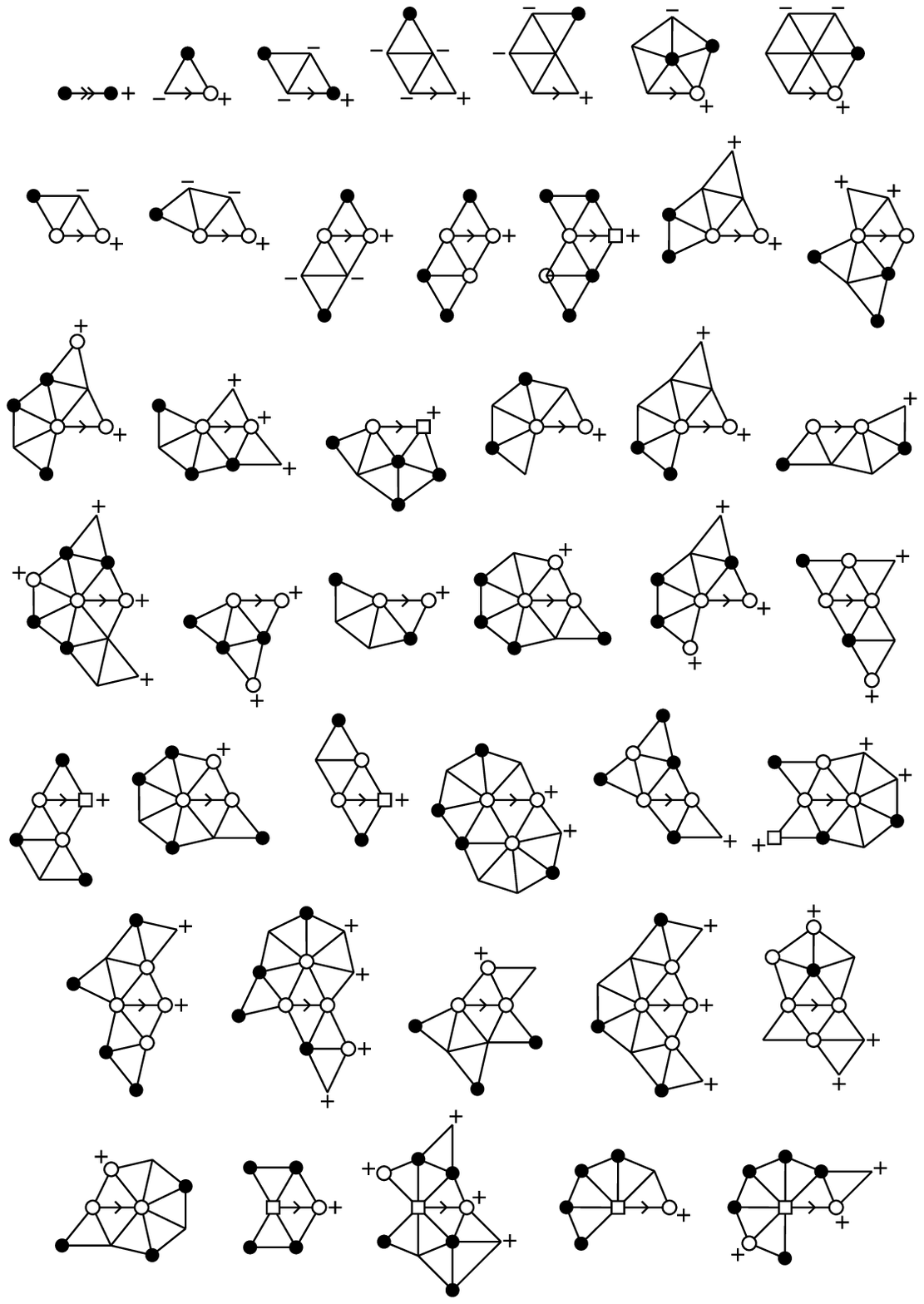}}}
\end{pspicture}
\renewcommand{\captionlabeldelim}{:}\caption{Our set of discharging rules, $\mcal{L}_{42}$.}\renewcommand{\captionlabeldelim}{}
\label{mainrules}
\end{center}
\end{figure}
}

\newcommand{\wellpositionedfig}{
\renewcommand{\height}{3}
\usualfigureheader
\psset{xunit=0.7cm, yunit=0.7cm}
\rput(-3.8, 1){
\rput(0.5,-1.5){\includegraphics{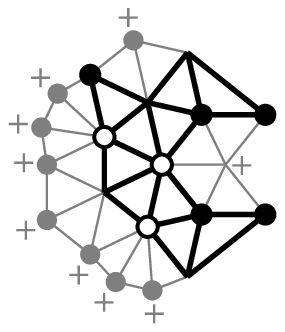}}}
\end{pspicture}
\renewcommand{\captionlabeldelim}{:}\caption{The configuration in bold appears in every cartwheel fitting the drawn part but is never well-positioned. Vertices with $\gamma_P^+(v) = \infty$ are lightened for clarity.}\renewcommand{\captionlabeldelim}{}
\label{wellpositioned}
\end{center}
\end{figure}
}

\newcommand{\spanningtreefig}{
\renewcommand{\height}{1.5}
\usualfigureheader
\psset{xunit=0.7cm, yunit=0.7cm}
\rput(0,1){\psline[linewidth=1.3pt,arrowsize=0.22]{->}(-0.8,0)(0.8,0)}
\rput(-7, -0.1){
\rput(0.3,-1.5){\includegraphics{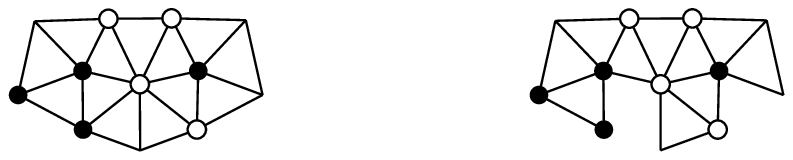}}}
\end{pspicture}
\renewcommand{\captionlabeldelim}{:}\caption{Preparing a spanning tree of faces.}\renewcommand{\captionlabeldelim}{}
\label{spanningtree}
\end{center}
\end{figure}
}

\newcommand{\bigradiusesfig}{
\renewcommand{\height}{1.5}
\usualfigureheader
\psset{xunit=0.7cm, yunit=0.7cm}
\rput(-9.4, -0.1){
\rput(0.3,-1.5){\includegraphics{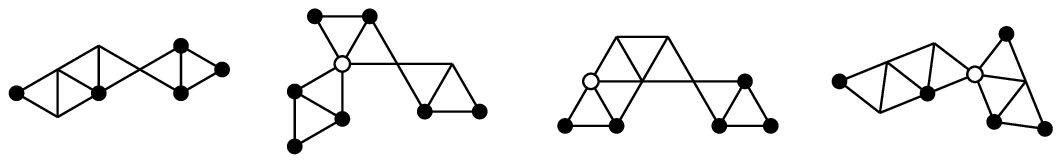}}}
\end{pspicture}
\renewcommand{\captionlabeldelim}{:}\caption{Some $D$-reducible configurations of radius greater than two that fit in a second neighborhood.}\renewcommand{\captionlabeldelim}{}
\label{bigradiuses}
\end{center}
\end{figure}
}

\newcommand{\unencodablefig}{
\renewcommand{\height}{1.3}
\usualfigureheader
\psset{xunit=0.7cm, yunit=0.7cm}
\rput(-2.3, 0.8){
\rput(0.3,-1.5){\includegraphics{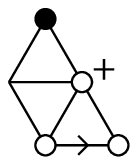}}}
\end{pspicture}
\renewcommand{\captionlabeldelim}{:}\caption{A rule that cannot be encoded as a part.}\renewcommand{\captionlabeldelim}{}
\label{unencodable}
\end{center}
\end{figure}
}

\begin{document}

\title{An Unavoidable Set of $D$-Reducible Configurations}
\author{John P. Steinberger}
\maketitle

\begin{abstract}
We give a new proof of the four-color theorem by exhibiting an unavoidable
set of \nc $D$-reducible configurations. The existence of such a set had
been conjectured by several researchers including Stromquist \cite{stromquist}, Appel and
Haken \cite{ah1} and Robertson, Sanders, Seymour and Thomas \cite{rsst}.
\end{abstract}

\section{Introduction}

The four-color theorem states that every map of simply connected countries can be colored in four colors such that countries with a common boundary segment are given different colors. Equivalently, one can say that every planar graph is vertex 4-colorable. The statement was conjectured in 1852 but remained open until 1976, when Appel and Haken \cite{ah1, ah2} announced their proof. Subsequently Robertson, Sanders, Seymour and Thomas \cite{rsst} 
obtained a substantially more compact proof. Both proofs are computer-aided.

In this paper we give a new proof similar to the one of Robertson et al. Our proof is technically simpler than both existing proofs because we use only the most basic type of reducible configuration, known as ``$D$-reducible''. However the proof is still computerized and, in fact, requires more computer time because it uses more configurations to make up for the loss of non-$D$-reducible configurations. 

All existing proofs of the four-color theorem follow an argument outlined by Heesch \cite{heesch}. The strategy is to exhibit an ``unavoidable set'' of ``reducible configurations''. Reducible configurations are local arrangements of countries that cannot appear in a smallest counterexample because their presence in a map implies the map can be colored from a smaller map by induction. A set of configurations is ``unavoidable'' if every map contains at least one configuration from the set. The unavoidable set of Appel and Haken has 1476 elements and the unavoidable set of Robertson et al. has 633 elements. Our set has \nc elements.
b
Reducible configurations come in two main flavors: $C$- and $D$-reducible configurations. Testing a configuration for $D$-reducibility is easy, but $C$-reducibility is more delicate. To show a configuration is $C$-reducible one needs to find a ``reducer'' that is safe to use in the induction step. There are a large number of possible reducers to consider and there is no single established safety test. Robertson et al.~only use reducers that fit a specific  safety criterion, whereas Appel and Haken use a wider variety of reducers which lead to a wider variety of complications. There is no safety check necessary for $D$-reducible configurations, as there are no reducers. 

Over half the configurations in the unavoidable sets of Robertson et al$.$ and about one third the configurations in the unavoidable set of Appel and Haken are $C$-reducible. However both teams conjectured the existence of an unavoidable set of $D$-reducible configurations. This paper settles the conjecture. Stromquist \cite{stromquist} made the same conjecture a little before Appel and Haken \cite{ah1}.

Proving that a set of configurations is unavoidable requires the use of a ``discharging argument'', a technique invented by Heesch. Appel and Haken's discharging argument is specified by 487 discharging rules. The bulk of their proof is devoted to showing the correctness of the discharging rules; the task is immense and, apparently, no full independent check has ever been made. Robertson et al.~use a different type of discharging procedure, which enables them to have vastly fewer rules (only 32) but checking the procedure by hand is of the same order of tedium as for Appel and Haken. However, Robertson et al.~supply machine-readable proof scripts that allow the discharging procedure to be checked by a computer in a few minutes. These proof scripts are probably the most significant difference between Robertson et al.'s proof and Appel and Haken's (though there are other notable differences).

Our own discharging procedure is quite similar to Robertson et al.'s. We have a few more rules (42 instead of 32) but the essential features are the same. As a consequence, the correctness of our discharging procedure can be established with the same type of machine-readable script. The script is verified by Robertson et al.'s original program with some very minor modifications (for example the constant specifying the maximum recursion depth needs to be incremented). The second program Robertson et al.~left available, which is for checking configuration reducibility, can also be used to check our own set of configurations after a one-line change to accomodate larger configurations, so altogether no new programs are necessary for checking our proof. All one needs to do is swap Robertson et al.'s configurations and rules (contained in two separate files) for our own, and replace the proof scripts. 
\ifarxiv
The necessary files and programs are in the `aux' folder accomponaying this arxiv submission. 
\else
The necessary files and programs are available from can be found with the electronic copy of this article at \arxurl. 
\fi
The changes made to the programs are also documented there.

\begin{table}\label{ringsize}
\begin{center}
\begin{tabular}{|c|c|c|c|c|c|c|c|c|c|}
\hline
Ring Size & 6--8 & 9 & 10 & 11 & 12 & 13 & 14 & 15 & 16 \\ 
\hline
\% & 0.26 & 0.53 & 2.30 & 6.60 & 16.7 & 26.1 & 26.3 & 15.5 & 5.7 \\ 
\hline
\end{tabular}
\end{center}
\caption{Percentage breakdown of the unavoidable set of \nc configurations according to ring-size.}
\end{table}

The cost of testing a configuration for reducibility increases roughly by four-fold with its ``ring-size'' (the number of countries surrounding the configuration). Appel and Haken used only configurations of ring-size at most 14, which they correctly estimated would give them enough configurations to find an unavoidable set. The unavoidable set of Robertson et al.~also has only configurations of ring-size 14 or less. Our unavoidable set uses configurations of ring-size up to 16 (a configuration of ring-size 16 takes a few minutes on average to test, but the exact amount of time varies greatly). The ring-size percentage breakdown of our unavoidable set is shown in Table \ref{ringsize}. We were not able to find an unavoidable set using only configurations of ring-size 14 or less, but we have no opinion as to whether such a set exists or not.

Altogether we tested over 130'000 configurations during the course of our work. Of these, $42.8\%$ were $D$-reducible and about $90\%$ were either $C$- or $D$-reducible (from a random sample, eschewing\footnote{While ensuring the safety of reducers is a necessary and annoying aspect of $C$-reducibility, it is rare to loose a configuration because none of the reducers are safe.} safety tests on the reducers, and not counting non-minimally reducible configurations). By comparison, Robertson et al.~tested around 14'000 configurations, of which $92.7\%$ were either $D$- or $C$-reducible and $54.4\%$ were $D$-reducible. Thus the ``loss rate'' consisting of configurations tested that turn out not to be usable climbs from approximately $10\%$ to $50\%$ when $C$-reducible configurations are disallowed.

In a recent tour de force, Gonthier \cite{gonthier} announced a formal language implementation of the Robertson et al.~proof. In this setting, the terms ``simple map'' and ``four-colorable'' are first formalized from axioms. Then the statement ``\texttt{M simple map} $\implies$ \texttt{M four-colorable}'' is verified by a proof checker
which reads through 60'000 lines of formal language proof. In order to trust the theorem has been proved, however, one does not need to check the 60'000 lines; one only needs to check the file containing the definitions, the file containing the statement, and to trust the proof checker 
itself. This constitutes an improvement over traditional proofs because the proof-checker verifies
all details of the proof, including manual combinatorial arguments that are only sketched or left entirely to the reader in the original proof. 
Presumably, since our proof is structurally identical to Robertson et al.'s, it should also be possible to transcribe it formally. 

Our reducibility computations use boolean closures as opposed to the more powerful ``block-count'' closures \cite{cohen, swart}. Thus all our configurations are $D$-reducible in the classical sense of Heesch \cite{heesch}, which fits the purpose of having the proof as non-technical as possible. This is also a stronger result since any boolean-$D$-reducible configuration is block-count-$D$-reducible but not vice-versa. Appel and Haken and Robertson et al.~do not use any block-count reducible configurations in their unavoidable sets either (the concept was not even known when Appel and Haken published). Robertson et al.~did investigate the possibility, however, and advertised an alternate unavoidable set of only 591 configurations containing some block-count reducible configurations.

Section 2 of the paper formulates the graph-theoretic problem and gives basic results on minimal counterexamples; Section 3 introduces the notion of a configuration and of $D$- and $C$-reducibility; Section 4 gives the discharging procedure and its proof of correctness. Part of the purpose of the paper is to give an account of the proof that is as self-contained and as accessible as possible. 

\section{Minimal Counterexamples}

The graph-theoretic formulation of the problem is obtained from a map by creating a vertex for every country and joining vertices of adjacent countries by an edge, with one edge added for every common border between two countries (there may be several). The resulting graph is called the \emph{dual} of the map, is planar, and can be vertex 4-colored if and only if the original map is four-colorable. The dual is loopless or else could not be colored at all. In map-theoretic language, this means that a country is not allowed to have a border with itself. 
Giving precise topological definitions of the terms ``map'', ``country'', ``border'' and so on, and deducing the existence of a well-defined dual graph leads to complications that would take us afield \cite{fritsch, gonthier}. Instead we will take for granted that the four-color theorem reduces to the 4-colorability of loopless planar graphs (or one can simply adopt the latter formulation as the statement to be proved). 

A \emph{minimal counterexample} is a simple (no multiple edges) loopless planar graph that is not 4-colorable but such that every loopless planar graph with fewer vertices is 4-colorable. Clearly, if the four-color theorem is false then minimal counterexamples must exist, so if minimal counterexamples do not exist then the theorem is true. 

One can assume that a minimal counterexample is a triangulation, as triangulating the faces of a planar graph only makes the graph harder to color. Or one can prove that, in fact, any minimal counterexample must be a triangulation: if the graph has a face of 4 or more sides then there is always a pair of non-adjacent vertices on the boundary of the face (Jordan curve theorem); identifying these vertices yields a loopless graph with fewer vertices whose coloring gives a coloring for the original graph. It is equally easy to see that a minimal couterexample cannot have any vertices of degree less than four, as one could color the graph by removing and then reinserting these vertices.

Our next, less trivial observation about minimal counterexamples is due to Birkhoff \cite{birkhoff}. A \emph{circuit} $C$ is a closed, non-self-intersecting walk on the edges of a graph. We write $E(C)$ for the set of edges of a circuit.

\begin{definition}\label{shortcircuitdef}
A \emph{short circuit} of a planar graph $G$ is a circuit $C$ with $|E(C)| \leq 5$ such that the two open regions bounded by $C$ contain at least one vertex each of $G$ if $|E(C)| \leq 4$ and at least two vertices each if $|E(C)| = 5$.
\end{definition}

\begin{lemma}[Birkhoff]\label{int6a}
A minimal counterexample has no short circuits. 
\end{lemma}

\noindent Thus, for example, a minimal counterexample $G$ cannot contain a vertex of degree 4, since the 4-circuit of edges around a vertex of degree 4 would be a short circuit of length 4 in $G$ (checking that the two regions defined by the 4-circuit both contain vertices of $G$ 
uses the fact that $G$ is a simple triangulation of minimum degree 4). Kempe \cite{kempe} had already shown that a minimal counterexample cannot have a vertex of degree 4.

The import of Lemma \ref{int6a} is best conveyed by an alternate characterization of the result also due to Birkhoff. The \emph{distance} between two vertices $u, v$ of a connected graph $G$ is the length of the shortest path between them (zero if $u = v$) and the \emph{second neighborhood} of a vertex $v$ is the subgraph induced by vertices at distance at most 2 from $v$. 
In an arbitrary triangulation the second neighborhood of a vertex may look different than two neatly concentric triangular strips. For example 
a second neighborhood can disconnect its complement (this could even be effected by the first neighborhood).
But Birkhoff observed that the absence of short circuits implied well-structuredness of second neighborhoods in the sense of the following definition:

\begin{definition}
The second neighborhood of a vertex $v$ of a planar triangulation is \emph{well-behaved} if the subgraph of $G$ induced by vertices at distance $i$ from $v$ is a circuit for $i = 1, 2$.
\end{definition}

\begin{lemma}[Birkhoff]\label{nn}
The vertices of a minimal counterexample have well-behaved second neighborhoods.
\end{lemma}

\noindent Lemma \ref{int6a} implies Lemma \ref{nn}. Lemma \ref{nn} also implies Lemma \ref{int6a} if one assumes the knowledge that minimal counterexamples have no vertices of degree 4, which can be demonstrated independently of Lemma \ref{int6a}.

\begin{proof}[Proof of Lemma \ref{nn} from Lemma \ref{int6a}]
Let $v$ be a vertex of a minimal counterexample $G$ with no short circuits. Thus $G$ has no vertices of degree less than 5. 
Let $d$ be the degree of $v$ and let 
$e_1, \ldots, e_d$ be a clockwise ordering of the edges around $v$. Let $u_1, \ldots, u_d$ be the endpoints of $e_1$, $\ldots$, $e_d$. Because $G$ is loopless and simple, $v, u_1, \ldots, u_d$ are distinct. Because $G$ is a triangulation, $u_i, u_{i+1}$ are adjacent for $1 \leq i \leq d$ (setting $u_{d+1} = u_1$).
Moreover there are no other adjacencies between the $u_i$'s or else $G$ would have a short circuit of length 3. Therefore the vertices at distance 1 from $v$ induce a circuit of length $d$ in $G$.

Let $d_1, \ldots, d_d$ be the degrees of $u_1, \ldots, u_d$ in $G$. For $1 \leq i \leq d$, let $w_{i,1}, \ldots, w_{i,d_i-3}$ be the $d_i - 3$ vertices adjacent to $u_i$ not in the set $\{v, u_1$, $\ldots$, $u_d\}$, and labeled such that $w_{i,1}, \ldots, w_{i,d_i-3}, u_{i+1}$, $v$, $u_{i-1}$ is a clockwise list of the vertices adjacent to $u_i$ (with indices having values in the set $\{1, \ldots, d\}$ taken cyclically).
Then $w_{i,d_i-3} = w_{i+1,1}$ because $G$ is a triangulation, and $w_{i,j}$ is adjacent to $w_{i,j+1}$ for $1 \leq j < d_i - 3$ for the same reason. 
Let $S$ be the sequence of vertices $w_{1,1}, \ldots, w_{1,d_1-4}, w_{2,1}, \ldots, w_{2,d_2-4}, \ldots, w_{d,1}, \dots, w_{d_d-4}$ ($S$ is nonempty because $d_i \geq 5$ for all $i$). Then $S$ contains all vertices at distance 2 from $v$ (because every neighbor of a $u_i$ has been accounted) and the elements of $S$ are cyclically adjacent as just observed.
One can easily check that the vertices of $S$ are distinct from the fact that $G$ has no short circuits of length $\leq 4$, and that there are no other adjacencies in $S$ besides the cyclic adjacencies from the fact that $G$ has no short circuits of length $\leq 5$ (details omitted). Thus the vertices at distance 2 from $v$ also induce a circuit in $G$.
\end{proof}

\noindent We sum up our knowledge of minimal counterexamples with the following definition:

\begin{definition}
A triangulation is \emph{internally 6-connected} if it has no short circuit and has minimum degree 5.
\end{definition}

\noindent Thus one can restate Lemma \ref{int6a}:

\begin{lemma}[Birkhoff]\label{int6}
A minimal counterexample is an internally 6-con-nected triangulation.
\end{lemma}

\noindent To keep the paper self-contained we finish this section with a proof of Lemma \ref{int6}, or equivalently of Lemma \ref{int6a}. The ideas used---Kempe chain arguments---recur in Section 3, so beginners should at least read part of the proof.

\begin{proof}[Proof of Lemma \ref{int6}] Let $G$ be a minimal counterexample. We wish to show that $G$ cannot contain a short circuit of length 3, 4 or 5. The fact that $G$ cannot contain a short circuit of length 3 is trivial, since if $G$ has a short 3-circuit one can color the graph consisting of $G$ minus the interior of the 3-circuit and the graph consisting of $G$ minus the exterior of the 3-circuit separately, and recover a coloring for $G$ by permuting one of the colorings to match the other. 

To see that $G$ cannot contain a short circuit of length 4 let $a$, $b$, $c$, $d$ be the vertices in order of a short circuit of length 4 of $G$. Let $\gA$, $\gB$ be the graphs obtained by deleting respectively the interior and exterior of the circuit $abcd$. One can color $\gA$ and $\gB$ by induction, but there is no guarantee the two the colorings induced on $abcd$ will be equivalent under permutations. In fact there are four equivalence classes under permutations for colorings of the 4-circuit; if the colors are $\{\cR, \cG, \cB, \cY\}$, representatives for the four classes are \cR\cG\cB\cY, \cR\cG\cR\cY, \cR\cG\cB\cG{} and \cR\cG\cR\cG.

Let $\gA'$ be the graph obtained from $\gA$ by adding an edge between $a$ and $c$; then $\gA'$ is a loopless planar graph 
because $a$ and $c$ are distinct (definition of a circuit). Since $\gA'$ has fewer vertices than $G$ it is 4-colorable, and induces a coloring on the circuit $abcd$ where $a$ and $c$ have distinct colors. Hence $\gA$ either admits a 4-coloring where $abcd$ are colored \cR\cG\cB\cY{} or admits a coloring where $abcd$ are colored \cR\cG\cB\cG. One could instead add an edge between $b$ and $d$, which would lead to the conclusion that $\gA$ either admits a 4-coloring where $abcd$ are colored \cR\cG\cB\cY{} or admits a 4-coloring where $abcd$ are colored \cR\cG\cR\cY. Thus either the coloring \cR\cG\cB\cY{} extends to $\gA$ or else both the colorings \cR\cG\cB\cG{} and \cR\cG\cR\cY{} extend to $\gA$.

As the same observation applies to $\gB$, we can assume without loss of generality that \cR\cG\cB\cY{} extends to $\gA$ and that \cR\cG\cB\cG{} and \cR\cG\cR\cY{} extend to $\gB$, or else $\gA$ and $\gB$ induce some common coloring on $abcd$. Now consider a coloring of $\gA$ where $abcd$ are colored \cR\cG\cB\cY{}. There may or may not be a path of $\{\cR,\cB\}$-vertices from the vertex $a$ (which is \cR) to the vertex $c$ (which is \cB) in $\gA$. Assume first that there isn't. Then the maximal connected component of \cR-\cB{} vertices attached to $c$ does not include $a$. If we switch the colors \cR{} and \cB{} inside of the component we obtain another valid coloring of $\gA$
in which the colors of $a$, $b$ and $d$ remain \cR, \cG{} and \cY, but the color of $c$ is switched to \cR, so the coloring \cR\cG\cR\cY{} extends to $\gA$, which settles the case because that coloring also extends to $\gB$. 

In the other case there is a path of \cR-\cB{} vertices from $a$ to $c$ in $\gA$. In this case 
there cannot be a \cG-\cY{} path from $b$ to $d$ (Jordan curve). Therefore we can switch the colors \cG{} and \cY{} in, say, the maximal connected component of \cG-\cY{} vertices containing $d$ without affecting $b$, and obtain the non-equivalent coloring \cR\cG\cB\cG{} of $abcd$. But this coloring extends to $\gB$, so we are again done. Thus a minimal counterexample cannot contain a short 4-circuit.

Finally to see that $G$ cannot contain a short circuit of length 5 let $abcde$ be a short circuit of length 5 in $G$, and again let $\gA$ and $\gB$ be the graphs obtained by deleting respectively the interior and exterior of the 5-circuit $abcde$ from $G$. Both $\gA$ and $\gB$ can be 4-colored by induction, and each 4-coloring induces a 4-coloring of the 5-circuit $abcde$. There are only two different equivalence classes under color permutations of colorings of the 5-circuit up to rotation---representatives are \cR\cG\cR\cG\cB{} and \cR\cG\cR\cY\cB---for a total of 10 different equivalence classes of colorings of the 5-circuit. We let $\rA{1}, \dots, \rA{5}$ be the 5 equivalence classes whose representatives are the cyclic permutations of \cR\cG\cR\cG\cB{} and $\rB{1}, \dots, \rB{5}$ the 5 equivalence classes whose representatives are the cyclic permutations of  \cR\cG\cR\cY\cB{} (with the cyclic shifts being arranged such that, say, $\rA{2}$ has representative \cB\cR\cG\cR\cG, etc). 

If $H$ is a graph with boundary $abcde$ we write $H|\rA{i}$ to mean that colorings from $\rA{i}$ can be extended to $H$, and likewise for $H|\rB{i}$. The conclusion will quickly follow from the following two observations, valid for $H = \gA, \gB$ (indices are taken cyclically):\\

	  {
	  \renewcommand{\theenumi}{\arabic{enumi}}
	  \renewcommand{\labelenumi}{(\theenumi)}
	  \begin{list}{2}{
	    \setlength{\itemindent}{23pt}
	    \setlength{\itemsep}{0pt}
	    \setlength{\parskip}{-100pt}
            \setlength{\leftmargin}{0pt}}
\vspace{-1ex}
\item[(\#1)] $H|\rA{i} \wedge \neg H|\rA{i+1} \implies H|\rB{i}$
\item[(\#2)] $\neg H|\rA{i+1} \wedge \neg H|\rB{i+1} \implies H|\rA{i}$
\vspace{-0.5ex}
           \end{list}
          }

\noindent To see $(\#1)$, assume for example that $\gH|\rA{1}$ but that $\neg \gH|\rA{2}$. In particular, the coloring \cR\cG\cR\cG\cB{} extends to $\gH$, but the coloring \cR\cG\cB\cG\cB{} does not. Fix a coloring of $\gH$ that induces \cR\cG\cR\cG\cB{} on $abcde$. If there is no \cR-\cB{} path from $a$ to $c$ in $\gH$ then a \cR-\cB{} color interchange on the \cR-\cB{} component containing $c$ yields the coloring \cR\cG\cB\cG\cB{}, a contradiction, so there must be a \cR-\cB{} path from $a$ to $c$. But then there
is no \cG-\cY{} path from $b$ to $d$, so by doing a \cG-\cY{} color interchange on the \cG-\cY{} component containing $d$ we get the coloring \cR\cG\cR\cY\cB{} from $\rB{1}$, as desired. 

To prove $(\#2)$ form a smaller graph by identifying the vertices $a$ and $c$ in $\gH$; the coloring induced on $abcde$ must give equal colors to $a$ and $c$, so colorings from one of the equivalence classes $\rA{1}, \rB{1}, \rA{5}$ must extend to $\gH$, which is equivalent to saying that $\neg \gH|\rA{1} \wedge \neg \gH|\rB{1} \implies \gH|\rA{5}$, as desired. This slips a technicality: if $a$ and $c$ are already adjacent in $\gH$ then identifying them makes a loop, and the theorem cannot be applied by induction (only loopless planar graphs can be 4-colored; if the reader feels this is unfair and that loops should simply be ignored, they should consider that the coloring of the smaller graph with $a$ and $c$ identified does not pull back to a valid coloring of $\gH$). The way around this case is to notice that if $a$ and $c$ are adjacent in $\gH$, then $G$ already contains a short circuit of length 3 or 4.

To finish the proof that $G$ cannot contain a short circuit of length 5, form a new graph from $\gA$ by adding a vertex in the middle of the circuit $abcde$, connected to all vertices of the circuit. The new graph is planar, loopless, and is smaller than $G$ because the interior of $abcde$ has at least two vertices in $G$. By induction it can be 4-colored, which shows that $\gA$ must have a coloring that uses only 3 colors for the circuit $abcde$, so $\gA|\rA{i}$ for some $i$. Since likewise $\gB|\rA{j}$ for some $j$, and since we can assume that there is no equivalence class whose colorings extend both to $\gA$ and $\gB$, there must be some $i$ such that $\gA|\rA{i}$ and $\neg \gA|\rA{i+1}$. Without loss of generality, say that $i = 1$. Then $(\#1)$ implies $\gA|\rB{1}$ and $(\#2)$ applied with $H = \gB$ and $i = 5$ implies $\gB|\rA{5}$. Since $\neg \gB|\rA{1}$ (because $\gA|\rA{1}$) one can apply $(\#1)$ again to get $\gB|\rB{5}$ and then $(\#2)$ again to get $\gA|\rA{4}$. Keeping on one reaches the conclusion that $\gA|\rA{2}$, a contradiction. This finishes the proof that a minimal counterexample cannot have a short circuit.
\end{proof}

\section{Reducible Configurations}

Birkhoff also investigated circuits of length 6 but concluded that similarly simple results wouldn't be forthcoming for circuits of that length.
Without going into details, Birkhoff convinced himself by finding two disjoint sets of colorings of the 6-circuit such that (i) each set was 
self-consistent with respect to 
Kempe chain arguments, and (ii) each set accomodated colorings imposed on the 6-circuit by replacing the interior of the circuit 
with any small\footnote{Birkhoff carried out his analysis to include every graph with 3 or fewer internal vertices (more precisely: any planar graph with an external 6-sided face and at most 3 internal vertices). Extending the analysis to include graphs with more than three internal vertices seemed unlikely to change the result, as the number of colorings that extend to the 6-circuit only tends to grow with the number of internal vertices.} 
graph, including all combinations for identifying vertices of the 6-circuit or adding chordal edges to the 6-circuit. 
Thus, it seemed, two arbitrary graphs $\gA$ and $\gB$ bounded by a common 6-circuit could \emph{a priori} induce two disjoint sets of colorings on the circuit. 

\birkhoffdiamondfig

Instead of working with two unknown graphs on either side of the circuit Birkhoff then tried fixing one of the graphs. He found a pair of graphs $(\gA, \gA')$ with 6-sided external faces where $\gA'$ had fewer vertices than $\gA$ and such that any graph $\gB$ bounded by a 6-circuit and having at least one coloring of the 6-circuit in common with $\gA'$ 
also had a coloring in common with $\gA$ (this uses a Kempe chain argument). Thus no minimal counterexample could contain the graph $\gA$, or else one could replace $\gA$ by $\gA'$ and argue the existence of a coloring by induction. 
Birkhoff's graph $\gA$ and its ``reducer'' $\gA'$ are shown in Fig$.$ \ref{birkhoffdiamond}. The graph $\gA$ is known as the ``Birkhoff diamond'', and is the first ``reducible configuration'' discovered (it is also the smallest).

To correctly apply Birkhoff's argument one must show 
that the smaller graph with $\gA$ replaced by $\gA'$ is loopless. One can do this using short circuits. However, a different argument which does not use $\gA'$ avoids the technicality entirely. Indeed, one can show that any graph $\gB$ bounded by a 6-circuit must have a coloring of the 6-circuit in common with $\gA$ just by a more systematic application of Kempe chain arguments (it is unknown whether Birkhoff knew this). 
The difference between the two types of arguments is precisely the difference between $C$- and $D$-reducibility, with the latter argument being $D$-reducibility. When $D$-reducibility fails one must resort to $C$-reducibility, and endure the looplessness check on the reducer.

To go further we must systematize our understanding of Kempe chain arguments. The key notion is that of a \emph{consistent set} of colorings of a circuit. Broadly put, a Kempe chain argument shows that if a coloring $\rC$ of a circuit extends to a graph $G$ bounded by the circuit then there are sets of colorings 
$\mcal{C}_1$, $\ldots$, $\mcal{C}_k$ 
of the circuit, independent of $G$, such that all the colorings from at least one of the $\mcal{C}_i$'s extend to $G$. If we list the elements of $\mcal{C}_i$ as $\rC_{i,1}, \ldots, \rC_{i,n_i}$, we can encode this observation as an implication
\eqn\label{implic}
\rC \implies (\rC_{1,1} \wedge \dots \wedge \rC_{i,n_1}) \vee \dots \vee (\rC_{k,1} \wedge \dots \wedge \rC_{k,n_k}).
\uneq
A set of colorings $\mcal{C}$ of a circuit is then called ``consistent'' if every such implication is satisfied by setting colorings in $\mcal{C}$ to ``true'' and colorings not in $\mcal{C}$ to ``false''. 

Let $R$ be a circuit graph. A \emph{signed path arrangement} for $R$ is a partition of $V(R)$ into subsets $V_1$, $\ldots$, $V_n$ such that $V_i$ is a subset of a connected component of $R\back V_j$ for all $i \ne j$, 
together with a function $s: \{ V_1^1, \dots, V_{n}^{m_n}\} \rightarrow \{ 0, 1\}$ 
where $V_i^1, \dots, V_{i}^{m_i}$ 
are the connected components induced by $V_i$ in $R$. A \emph{color partition} is a partition of the set $\{\cR, \cG, \cB, \cY\}$ into two sets of two elements each. If $\rC$ is a coloring of $R$, $P = (V_1, \ldots, V_n, s)$ is a signed path arrangement for $R$, and $\theta = \{ \{ \rX, \rY \}, \{ \rW, \rZ \} \}$ is a color partition, then we say that 
\rC{} (resp$.$ $P$) $\theta$-fits $P$ (resp. \rC) 
if (i) the set of maximally connected subgraphs of $R$ colored only with $\rX$ and $\rY$ or only with $\rW$ and $\rZ$ is equal to the set 
$\{ V_1^1, \dots, V_{n}^{m_n}\}$ and (ii) $s(V_{i}^{j_1}) = s(V_{i}^{j_2})$ if and only if the first color of $V_{i}^{j_1}$ is equal to the first color $V_{i}^{j_2}$ as encountered when traveling clockwise around $R$. We simply say that ``\rC{} fits $P$'' if $\rC$ $\theta$-fits $P$ for some $\theta$. A set of colorings $\mcal{C}$ of $R$ is \emph{consistent} if for every $\rC \in \mcal{C}$ and every color partition $\theta$ there is a signed path arrangement $P$ that $\theta$-fits $\rC$ such that every coloring of $R$ that fits $P$ is in $\mcal{C}$.

It is easy to see that the set of colorings induced by a planar graph $G$ on a circuit $R$ bounding a region of $G$ is consistent. 
To check this, take a coloring $\rC$ of $R$ that extends to $G$ and a color partition $\theta = \{ \{ \rX, \rY \}, \{ \rW, \rZ \} \}$. 
Fix a 4-coloring of $G$ equal to $\rC$ on $R$ and declare vertices $u, v$ of $R$ to be in the same set $V_i$ if and only if there is an $\rX\rY$-path or a $\rW\rZ$-path from $u$ to $v$ in $G$ 
(in particular, $u$ and $v$ must both be colored $\rX$ or $\rY$ or must both be colored $\rW$ or $\rZ$).
Say that $n$ different sets $V_1, \ldots, V_n$ are obtained this way. 
Because $G$ is planar, each $V_i$ is a subset of the same component of $R\back V_j$ for all $j \ne i$.
Define $s$ such that $\rC$ $\theta$-fits $P = (V_1, \ldots, V_n, s)$, in the obvious way (there is more than one $s$ that works; any will do). 
We need to check that every coloring of $R$ that fits $P$ extends to $G$. Let $H_1, \ldots, H_n$ be the $\rX\rY$- or $\rW\rZ$-components of $G$ containing $V_1, \ldots, V_n$ respectively. Making a color interchange (switching $\rX$ and $\rY$, or $\rW$ and $\rZ$) in any $H_i$ results in a valid coloring of $G$ whose restriction to the circuit $R$ again $\theta$-fits $P$, and such that the $\rX\rY$- and $\rW\rZ$-components of $G$ containing $V_1, \ldots, V_n$ are still $H_1, \ldots, H_n$. By repeating the process one can thus make any combination of color interchanges to the $H_i$'s, and it is easy to see that the colorings of $R$ generated this way and by permutations of $\{ \rX, \rY, \rZ, \rW \}$ are all the colorings that fit $P$.

It is useful to keep in mind the order of the quantifiers that define whether a set $\mcal{C}$ is consistent: for \emph{every} coloring $\rC \in \mcal{C}$ and \emph{every} color partition $\theta$ there is \emph{some} signed path arrangement $P$ that $\theta$-fits $\rC$ such that \emph{every} coloring that fits $P$ is in $\mcal{C}$. Thus for every coloring \rC{} there are three implications like \eqref{implic}: one for each color partition. The disjunction is taken over all signed path arrangements that fit \rC, with each implicant being the conjunction of the colorings that fit the path arrangement.

In practice one most often wants to know whether a given set $\mcal{C}_1$ of colorings of a $k$-circuit has a nonempty consistent subset or not, and, if so, what is the biggest consistent subset (because the union of two consistent sets is consistent, there is a unique maximal consistent subset). To do this one first computes the set $\mcal{P}_1$ of signed path arrangements such that $P \in \mcal{P}_1$ if and only if every coloring that fits $P$ is in $\mcal{C}_1$. Then one computes the set $\mcal{C}_2$ of colorings 
consisting of all $\rC \in \mcal{C}_1$ 
that $\theta$-fit a signed path arrangement in $\mcal{P}_1$ for each color partition $\theta$. Then one constructs $\mcal{P}_2$ from $\mcal{C}_2$ like $\mcal{P}_1$ is constructed from $\mcal{C}_1$, and so forth until $\mcal{C}_i = \mcal{C}_{i-1}$ or until $\mcal{C}_i = \emptyset$. (This shows the idea, but does not describe the reducibility program we actually used, which is the one of Robertson et al. Indeed, Robertson et al.~do all their reducibility computations using Tait's dual formulation of vertex 4-colorability in terms of edge 3-colorability. Readers should consult their paper and the note accompanying their program \cite{reduce} for further details.)

For example if $\mcal{C}$ is the set of 4-colorings that the Birkhoff diamond (the graph on the left of Fig$.$ \ref{birkhoffdiamond}) induces on its 6-circuit and 
$\mcal{C}^*$ is the set of all colorings of the 6-circuit, then the above algorithm shows that $\mcal{C}^*\back \mcal{C}$ has no nonempty consistent subset. Thus any graph obtained by gluing the Birkhoff diamond to
a 6-circuit bounding the face 
 of a planar graph $G$ cannot be a minimal counterexample, because $G$ will be 4-colorable by induction and the set of colorings that $G$ induces on the 6-sided face, being consistent, must include some element of $\mcal{C}$, leading to a coloring of the entire graph. 

Not every face of a connected planar graph $G$ is bounded by a circuit, given that $G$ could have cut-vertices or bridges (edges with the same face on either side), but it is useful to glue graphs bounded by circuits to those faces as well. To glue a circuit $R$ to an arbitrary face $f$ one simply traces a walk around the boundary of $f$ in the natural way. The length of $R$ must be equal to the number of edges plus the number of bridges incident with $f$. The walk defines a map $\phi: V(R) \rightarrow V(G)$.
We say that $\phi$ 
\emph{wraps $R$ around $f$}. 
Every coloring of $G$ induces a coloring of $R$ by setting the color of $v \in V(R)$ to the color of $\phi(v) \in V(G)$. We call this coloring of $R$ the \emph{lift} of the coloring by $\phi$. The set of 4-colorings of $R$ that are obtained in this way is consistent, as can easily be seen from the same argument as above. 
In conclusion, we have:

\begin{lemma}\label{conslemma}
Let $G$ be a planar graph, $f$ be a face of $G$, and let $\phi$ map a circuit $R$ around $f$. Let $\mcal{C}$ be the set of lifts by $\phi$ of 4-colorings of $G$. Then $\mcal{C}$ is consistent.
\end{lemma}

\noindent We have not yet given a precise meaning to the term ``configuration''. 
A \emph{near-triangulation} is a connected planar graph in which all regions are triangles except possibly the infinite region. Following Robertson et al., a \emph{configuration $K$} is a near-triangulation $G(K)$ and a map $\gm_K : V(G(K)) \rightarrow \nn$ with the following properties:

	  {
	  \renewcommand{\theenumi}{\arabic{enumi}}
	  \renewcommand{\labelenumi}{(\theenumi)}
	  \begin{list}{3}{
	    \setlength{\itemindent}{35pt}
	    \setlength{\itemsep}{0pt}
	    \setlength{\parskip}{-5pt}
            \setlength{\leftmargin}{0pt}}
\item[(i)] for every vertex $v$, $G(K)\backslash v$ has at most two components, and if there are two then $\gm_K(v) = d(v) + 2$ where $d(v)$ is the degree of $v$ in $G(K)$,

\item[(ii)] for every vertex $v$, if $v$ is not incident with the infinite region, then $\gm_K(v) = d(v)$, and otherwise $\gm_K(v) > d(v)$; and in either case $\gm_K(v) \geq 5$,

\item[(iii)] $K$ has ring-size $\geq$ 2, where the \emph{ring-size} of $K$ is defined to be $\sum_v (\gm_K(v) - d(v) - 1)$, summed over all vertices $v$ incident with the infinite region such that $G(K)\back v$ is connected.
           \end{list}
          }

\noindent 
A configuration is meant to be embedded in a bigger graph, and the meaning of the various components of the definition is fairly transparent: $\gm_K$ specifies the degree of a vertex in the larger graph; the ring-size is the arity of the face left when $K$ is deleted from the larger graph; a configuration may have one or several cut-vertices, and these must be adjacent to exactly two vertices outside the configuration. The embedding notion is made precise below.

For the Birkhoff diamond, $G(K)$ is the graph on four vertices consisting of two adjacent triangles and $\gamma_K$ is identically 5. To draw a configuration one usually uses vertex shapes (an idea of Heesch) to show the values of $\gamma_K$, which is possible because in practice $\gamma_K \leq 11$. The list of vertex shapes is shown in Fig$.$ \ref{vertexshapes}. Some configurations are shown in Figs$.$ \ref{fivefivepair}, \ref{bernhart}, \ref{minorconfs} and \ref{proofconfs}. The first configuration of Fig$.$ \ref{minorconfs} is the Birkhoff diamond\footnote{We have been referring to the graph on the top left of Fig$.$ \ref{birkhoffdiamond} as the ``Birkhoff diamond'' but this is somewhat unorthodox; the first configuration of Fig$.$ \ref{minorconfs} is the Birkhoff diamond, and the graph of Fig$.$ \ref{birkhoffdiamond} is the \emph{free completion} of the Birkhoff diamond.}. 

\vertexshapesfig

The \nc configurations of our unavoidable set require over 30 pages to draw at a reasonable size. Instead of appending them as a figure to the paper, which would be of limited use, they are available online in both machine- and human-readable forms. For the rest of the paper we let $\mcal{U}_{\ncmath}$ denote the set of configurations in the file {\tt U\_{\ncmath}.conf} and drawn in {\tt U\_{\ncmath}.pdf}. 
\ifarxiv
The format of ${\tt U\_\ncmath.conf}$ is explained in the `README' file.
\else
These files are available with the electronic copy of this article at \arxurl. 
The format of ${\tt U\_\ncmath.conf}$ is also explained there.
\fi

Let $K$ be a configuration. A near-triangulation $S$ is a \emph{free completion of $K$ with ring $R$} if:

	  {
	  \renewcommand{\theenumi}{\arabic{enumi}}
	  \renewcommand{\labelenumi}{(\theenumi)}
	  \begin{list}{3}{
	    \setlength{\itemindent}{44pt}
	    \setlength{\itemsep}{0pt}
	    \setlength{\parskip}{-5pt}
            \setlength{\leftmargin}{0pt}}

\item[(i)\hfill] $\!\!\!\!\!$ $R$ is a circuit of length $\geq 2$ of $S$ and bounds the infinite region of $S$,

\item[(ii)\hfill] $\!\!\!\!\!$ $G(K) = S \back V(R)$ and $G(K)$ is an induced subgraph of $S$ and

\item[(iii)\hfill] $\!\!\!\!\!$ every vertex of $S$ not in $V(R)$ has degree $\gm_K(v)$ in $S$.
\end{list}
}

One can easily check that $|V(R)|$ is the ring-size of $K$ as defined in point (iii) of the definition of a configuration. The condition that a cut-vertex $v$ of a configuration have $\gm_K(v) = d(v) + 2$ ensures that the free completion of a configuration is unique (up to planar homeomorphism), so we speak of ``the'' free completion of a configuration. Checking the existence of a free completion from the definition of a configuration is fairly straightforward.

\begin{definition}
Let $K$ be a configuration and let $S$ be the free completion of $K$ with ring $R$. Then $K$ is \emph{$D$-reducible} if the maximal consistent subset of $\mcal{C}^*\back \mcal{C}$ is empty, where $\mcal{C}^*$ is the set of all 4-colorings of $R$ and $\mcal{C} \subseteq \mcal{C}^*$ is the set of colorings of $R$ that extend to $S$.
\end{definition}

Let $G$, $H$ be two planar graphs. We say that $G$ and $H$ are \emph{spherically homeomorphic} if the point-at-infinity completions of $G$ and $H$ are homeomorphic (otherwise put: the stereographic projections of $G$ and $H$ onto a sphere are homeomorphic subsets of the sphere). 

\begin{definition}\label{appearsdef}
A configuration $K$ \emph{appears} in a triangulation $T$ if $G(K)$ is spherically homeomorphic to an induced subgraph of $T$ such that the image of every vertex $v \in V(G(K))$ has degree $\gm_K(v)$ in $T$.
\end{definition}

\noindent If we had embedded all our graphs on spheres from the start then Definition \ref{appearsdef} would seem more natural, but in that case we 
could not have referred to the ``infinite region'' so easily, as was practical for the definition of configurations and free completions. 

If a configuration $K$ appears in a triangulation $T$ the finite faces of $K$ are not necessarily mapped to faces of $T$ by the (not necessarily unique) spherical homeomorphism, but this is will be the case if $T$ is internally 6-connected (in particular, has no short-circuits of length 3) and $G(K)$ is not a triangle. 

If a configuration $K$ appears in an internally 6-connected triangulation $T$ then the free completion $S$ of $K$ is not necessarily a subgraph of $T$: some vertices from the ring could be identified in $T$. However it is easy
to check that removing the image of $G(K)$ from $T$ leaves a face $f$ whose arity is the ring-size of $K$, and that there is a map $\phi$ wrapping the ring $R$ of the free completion $S$ around $f$ such that a vertex $v \in V(G(K))$ is adjacent to a vertex $w$ from $R$ if and only if the image of $v$ in $T$ is adjacent to $\phi(w)$. Let $T'$ be the graph obtained by removing the image of $G(K)$ from $T$. Then a 4-coloring of $T'$ extends to $T$ if and only if the coloring induced on $R$ by the $T'$-coloring extends to $S$. If $T$ is a minimal counterexample then $T'$ is 4-colorable and the set of colorings induced on $R$ by colorings of $T'$ is consistent by Lemma \ref{conslemma}. Since none of the colorings induced on $R$ should extend to $S$ or else $T$ would be 4-colorable, $K$ cannot be $D$-reducible.
In other words:

\begin{lemma}\label{cannot}
A $D$-reducible configuration cannot appear in a minimal counterexample.
\end{lemma}

\noindent Thus if one can exhibit a set $\mcal{U}$ of $D$-reducible configurations such that some member of $\mcal{U}$ appears in every internally 6-connected triangulation, minimal counterexamples cannot exist. This is the property of our set $\mcal{U}_{\ncmath}$. Of course one must not forget to show that every configuration in $\mcal{U}_{\ncmath}$ is actually $D$-reducible, but this must be done by computer. 
\ifarxiv
The reader interested in running the computation should consult the README file.
\else
The reader interested in running the computation should download Robertson et al.'s program and our file from \arxurl. 
\fi
Verifying all configurations takes about 10 hours on a modern processor.

\fivefivepairfig

Generally speaking, configurations that have many internal vertices compared to the ring-size are more likely to be reducible because they accept more colorings. For example one can show that configurations having vertices adjacent to more than three ring vertices
are either not $D$-reducible or else are not minimally reducible: some subportion of the configuration is $D$-reducible. Thus all vertices of our configurations have $\gamma_K(v) \leq d(v) + 3$. This result was conjectured by Heesch \cite{heesch} and proved by Tutte and Whitney \cite{tutte} (ideas for the proof are already implicit in Shimamoto \cite{shimamoto}). Heesch also conjectured that minimally reducible configurations could not have ``hanging 5-5 pairs'', which are two adjacent vertices $v, w$ of degree 2 in $G(K)$ with $\gm_K(v) = \gm_K(w) = 5$  (Fig$.$ \ref{fivefivepair}). One can prove this for $D$-reducible configurations and show more generally that minimal $D$-reducible configurations cannot have blocks of size 3 or less (Fig$.$ \ref{fivefivepair}) or cut-vertices adjacent to more than two ring vertices
(such cut-vertices were excluded by our definition of a configuration for reasons of simplicity). These irreducibility results were obtained by several researchers independently in the early 1970's, see Stromquist \cite{stromquist} for an account. Conjecturally, these conditions are all necessary for $C$-reducibility as well.

It is not necessary to understand $C$-reducibility in order to understand our proof, but it seems fitting we give some description of it since part of the interest of our work lies in the difference between $D$- and $C$-reducibility. Let $K$ be a configuration with free completion $S$ and ring $R$. Let $\mcal{C}$ be the set of colorings of $R$ that extend to $S$, and let $\mcal{C}^*$ be the set of all colorings of $R$. We say that $K$ is \emph{$C$-reducible} if the maximal consistent subset of $\mcal{C}'$ of $\mcal{C}^*\back \mcal{C}$ is nonempty (meaning $K$ is not $D$-reducible) and there exists a graph $S'$ with fewer vertices than $S$ whose external face has arity equal to the ring-size of $K$ and such that the set of colorings $\mcal{C}''$ of the external face of $S'$ that extend to $S'$ does not intersect $\mcal{C}'$ 
(one makes the set of colorings of the external face of $S'$ correspond to a set of colorings $\mcal{C}''$ of $R$ by wrapping $R$ around the external face of $S'$). 

When using a $C$-reducible configuration $K$ to show that a graph $G$ where $K$ appears is not a minimal counterexample one replaces $S$ (embedded in $G$) with the graph $S'$. If the resulting graph is loopless one can color it by induction, which shows some coloring not in $\mcal{C}'$ extends to the graph $\gA$ obtained by removing $G(S)$ from $G$. But the set of colorings which extends to $\gA$ is consistent, and being nonempty and not a subset of $\mcal{C}'$ must intersect $\mcal{C}$, so a coloring of $G$ can be recovered. 
Loops can appear in fairly pernicious ways, given that vertices on the boundary of $S$ may have non-cyclic adjacencies in $G$, or may even be identified in $G$ (and in $S'$). The difficulty of checking that the smaller graph is loopless regardless of $G$ is proportional to how different $S'$ is from $S$. Robertson et al.~keep their looplessness check fairly simple by contracting at most four edges of $S$ to make $S'$. Appel and Haken used arbitrary constructions for $S'$, and hence had more to worry about.

\section{Discharging}

As we have stated, configurations tend to be reducible when the ring-size is small compared to the number of vertices of the configuration. 
This is favored by having vertices of small degree (quite intuitively, vertices of small degree ``defend better'' against coloring). In particular, configurations with many vertices of degree 5 and 6 are likely to be reducible. It was already observed by Kempe \cite{kempe} that every triangulation must contain a vertex of degree 5 or less, and the number of such vertices augments rapidly if a triangulation contains many vertices of degree greater than 6. Hence the hope that any triangulation has a reducible configuration somewhere.

Concerning the appearance of vertices of degree 5, one can more exactly show that
\eqn
\label{chargee}
\sum_{v\in V(T)} (6 - d(v)) = 12
\uneq
for any triangulation $T$ where $d(v)$ is the degree of $v$. This is easily obtained from Euler's formula and from the formulas $2E = 3F$ and $\sum_v d(v) = 2E$. Every triangulation of degree $\geq 5$ must thus have at least 12 vertices of degree 5, more if the triangulation contains vertices of degree $\geq 7$. 

One way to interpret \eqref{chargee} is to place a ``charge'' of $6 - d(v)$ on every vertex, so that the charges sum to 12. Then the presence of positive charge is somewhat correlated to the appearance of reducible configurations, but not exactly. The process of discharging consists in ``recalibrating'' the charge distribution via a series of rules that move the charge between the vertices so that, after the recalibration, the presence of positive charge on a vertex exactly coincides with the appearance of a reducible configuration. 
Since the total charge remains positive during the recalibration, some vertex must have positive charge and some reducible configuration must appear somewhere. 

To find a discharging procedure one starts with an incomplete set of rules and one identifies 
``overcharging situations'' where vertices have positive charge but have no reducible configurations in their neighborhoods. One then simply adds new discharging rules to move the charge away from these vertices in the hope of obtaining a more exact calibration. The art lies in choosing rules that improve rather than worsen the situation and in choosing rules that help several overcharging situations simultaneously, if possible. The process stops when there are no overcharging situations left. There is no theoretical guarantee that the process will stop, of course, otherwise one would not need to produce a set of discharging rules in the first place.

\mayerfig
\bernhartfig
\minorconfsfig

It is important to get the discharging procedure started off on the right foot. Our discharging rules are inspired from those of Robertson et al., who themselves start off with a discharging procedure of Mayer \cite{mayer}. 
Let \emph{minor vertices} be vertices that start with a charge $\geq 0$ (vertices of degree 5, 6) and \emph{major vertices} be vertices that start with a negative charge (vertices of degree $\geq 7$). 
The rules of Mayer, shown in Fig$.$ \ref{mayerrules}, remove the charge from vertices of degree 5 and place it on major vertices. After Mayer's discharging rules are applied 
minor vertices 
have zero charge unless they are in the neighborhood of a reducible configuration. 
The correctness of Mayer's rules depends on the configurations of Fig$.$ \ref{minorconfs}. 
In the case when $C$-reducible configurations are allowed the configuration of Fig$.$ \ref{bernhart} (``Bernhart's diamond'') also helps; with this configuration, the charge may stop traveling before encountering a major vertex that would otherwise receive charge. Mayer's discharging rules divide the charge into fractions of 1/10, which remains the unit of exchange for the remainder of the discharging procedure.

After Mayer's discharging rules are applied the overcharged vertices are those of degree 7 and 8 (it is generally hard to overcharge vertices of degree $\geq 9$). These overcharged vertices must then negotiate among themselves. Thus the discharging procedure is two-staged: in the first stage charge is removed from minor vertices and in the second stage charge is redistributed among major vertices (the distinction between the two stages is psychological; there is no notion of sequentiality in the application of discharging rules). The fact that a discharging rule may have both a source and a sink of degree 7 (say) allows a few simple rules to combine among themselves to achieve an overall complex effect. By contrast, Appel and Haken use a diffferent discharging paradigm in which discharging rules all have sources of degree 5 and sinks of degree $\geq 7$. In this case the method for obtaining the discharging procedure is a bit different: when an overcharged vertex is identified, one of the vertices of degree 5 responsible for overcharging the major vertex looks for a different major vertex to pass some of its charge to. Rules cannot combine, which results in a much larger number of rules. Appel and Haken have 487 rules, compared to 32 for Robertson et al. (we have 42). Allaire \cite{allaire} also used a two-stage discharging procedure built on Mayer's rules in his unpublished proof, but it seems he never completed the details.

\rsstrulesfig
\mainrulesfig

In our case we formally define a 
\emph{discharging rule} $L$ to be a tuple $(q, G(L)$, $s$, $t$, $\gm_L^-, \gm_L^+)$ where $q \in \qq$, $G(L)$ is a near-triangulation with no cut-vertices, $s$ and $t$ are adjacent vertices of $G(L)$ each at distance at most two from every vertex in $G(L)$, and 
$\gm_L^-, \gm_L^+$ are functions from $V(G(L))$ to the set $\{ 5, 6, 7, \ldots \} \cup \{\infty \}$ such that $\gm_L^- = \gm_L^+(v) = d(v)$ if $v$ is an internal vertex of $G(L)$ and $d(v) < \gm_L^+(v) \geq \gm_L^-(v) \geq d(v)$ if $v$ is not an internal vertex of $G(L)$, where $d(v)$ is the degree of $v$ in $G(L)$. To draw discharging rules we use the same vertex shapes as for configurations. For a vertex $v$ with $\gm_L^+(v) = \infty$ we put the shape for a vertex of degree $\gm_L^-(v)$ with a `+' next to the vertex. If $\infty > \gm_L^+(v) > \gm_L^-(v) = 5$ we put the shape for a vertex of degree $\gm_L^+(v)$ with a `$-$' next to the vertex. Otherwise $\gm_L^-(v) = \gm_L^+(v)$, except for 
rule 13 of Fig$.$~\ref{mainrules} that has an exceptional vertex with $(\gm_L^-(v), \gm_L^+(v)) = (6, 7)$, for which we use an empty circle\footnote{Rules 14 and 23 of Robertson et al. could also be combined by using this representation. Thus a fair comparison of the two sets of rules would either state the number of rules of Robertson et al.~as 31 or our number of rules as 43. However we did not want to waste figure space for our set, and did not want to create confusion by quoting the number of rules of Robertson et al. as 31 when it is quoted as 32 elsewhere.}. All our rules have $q \in \{\frac{1}{10}, \frac{1}{5} \}$. To show the placement of $s, t$ one puts an arrow on the edge from $s$ from $t$ (in that direction), with a double arrow when $q = \frac{1}{5}$ and a single arrow otherwise. See Figs$.$ \ref{rsstrules}, \ref{mainrules}.

Let $T$ be a triangulation and $H$ an induced subgraph of $T$. We say that a rule $L$ \emph{appears with image $H$, source $u$ and sink $w$ in $T$} if there is a spherical homeomorphism from $G(L)$ to $H$ taking $s$ to $u$ and $t$ to $w$ such that the image under the homeomorphism of each vertex $v$ in $G(L)$ has degree at least $\gm_L^-(v)$ and at most $\gm_L^+(v)$. 
If $u, w$ are adjacent in $T$ the \emph{total charge carried from $u$ to $w$ by $L$} is $q(L)$ times the number of distinct images $H$ in $T$ such that $L$ appears with image $H$, source $u$ and sink $w$ in $T$ (there can be more than one image $H$ because of reflections of the rule; but if $G(L)$ has mirror symmetry and $T$ is internally 6-connected there is at most one image). We put $r(u,w,L)$ for the total charge carried from $u$ to $w$ by $L$. 
We put $r_\mcal{L}(u,w) = \sum_{L \in \mcal{L}}r(u,w,L)$ for any finite set $\mcal{L}$ of discharging rules.

For a triangulation $T$, a set of rules $\mcal{L}$ and $u \in V(T)$ we define the \emph{charge $c_\mcal{L}(u)$ of $u$ with respect to $\mcal{L}$} to be
\eqn
\label{chargo}
c_\mcal{L}(u) = 6 - d(v) - \sum_w r_\mcal{L}(u,w) + \sum_w r_\mcal{L}(w,u)
\uneq
where the two sums are taken over all vertices adjacent to $u$. It follows directly from \eqref{chargo} and \eqref{chargee} that
\eqn
\label{sumcharge}
\sum_{u\in V(T)} c_{\mcal{L}}(u) = 12,
\uneq 
since every term $r_\mcal{L}(u,w)$ appears once positively and once negatively in the sum. Thus for any triangulation $T$ and any set of rules $\mcal{L}$ there is some $u \in V(T)$ such that $c_\mcal{L}(u) > 0$.

Fig$.$ \ref{mainrules} contains our rules. For comparison, Fig$.$ \ref{rsstrules} shows the rules of Robertson et al. The first seven rules of each set implement Mayer's discharging procedure; note there are slight differences in the first seven rules between our set and the set of Robertson et al., which are due to the loss of Bernhart's diamond (Fig$.$ \ref{bernhart}).

Some of our rules from the first line transfer charge between two vertices even when there is a reducible configuration present; for example the third rule if all the vertices have their smallest degree. This is harmless. One could split such rules into two or more rules of equivalent effect that would avoid such ``unnecessary'' charge transfers, but the resulting discharging procedure would not be any easier to check and the number of rules would be greater. 

Note that $c_\mcal{L}(u)$ only depends on the second neighborhood of $u$, because every vertex of a rule is at distance at most two from both the source and sink. 
Later this will greatly simplify the enumeration of 
cases for which $c_{\mcal{L}}(u) > 0$, given that second neighborhoods are well-behaved in internally 6-connected triangulations (Lemma \ref{nn}).
Another consequence is that reducible configurations must be sought within second neighborhoods, because only the second neighborhood determines the charge of a vertex. 
Thus all the configurations of $\mcal{U}_{\ncmath}$ are contained in the second neighborhood of a vertex, as is the case for the unavoidable set of Robertson et al. 
The existence of such a set was originally conjectured by Heesch \cite{heesch}. 

To show the correctness of our set of rules we must show that if $u$ is a vertex of an internally 6-connected triangulation with $c_{\mcal{L}_{42}}(u) > 0$ then 
the second neighborhood of $u$ contains a configuration from $\mcal{U}_{\ncmath}$. The proof is divided into two main parts: the cases when $5 \leq d(u) \leq 11$ and the case when $d(u) \geq 12$. For the first case, separate machine-readable proofs are written for each degree $d(u)$; these are essentially exhaustive analyses of the possible second neighborhoods of $u$. In the second case one shows that $r_{\mcal{L}_{42}}(v,u) \leq \frac{1}{2}$ for any neighbor $v$ of $u$ as long as no reducible configurations appear around $u$. Then $c_{\mcal{L}_{42}}(u) \leq 6 - d(u) + \frac{1}{2}d(u) \leq 6 - 12 + \frac{1}{2}12 = 0$, so $u$ is not overcharged.

While this gives the idea some precise definitions are missing. For example we have not defined what it means for a second neighborhood to ``contain'' a reducible configuration, even though this may seem intuitively obvious. In fact a second neighborhood, as defined, is an induced subgraph of the triangulation and carries no information about the degrees of its outer vertices, so it would be meaningless to define configuration appearance only with respect to this subgraph. What we need is a data structure that carries degree information about all vertices in the second neighborhood, whence the upcoming definition.

A \emph{cartwheel} $W$ is a near-triangulation $G(W)$ with a function $\gm_W: V(G(W)) \rightarrow \{5, 6, 7, \ldots \}$ such that $W$ contains two disjoint induced circuits $C_1$ and $C_2$ and a vertex $u$ not on $C_1$ or $C_2$, such that

	  {
	  \renewcommand{\theenumi}{\arabic{enumi}}
	  \renewcommand{\labelenumi}{(\theenumi)}
	  \begin{list}{3}{
	    \setlength{\itemindent}{42pt}
	    \setlength{\itemsep}{0pt}
	    \setlength{\parskip}{-5pt}
            \setlength{\leftmargin}{0pt}}

\item[(i)\hfill] $\!\!\!\!\!$ $V(G(W)) = \{u\} \cup V(C_1) \cup V(C_2)$,

\item[(ii)\hfill] $\!\!\!\!\!$ $u$ is adjacent to all vertices of $C_1$ and to no vertices of $C_2$,

\item[(iii)\hfill] $\!\!\!\!\!$ the set of vertices adjacent to the infinite region is $V(C_2)$, and

\item[(iv)\hfill] $\!\!\!\!\!$ $\gm_W(v) = d(v)$ for every internal vertex $v$ of $G(W)$, where $d(v)$ is the degree of $v$ in $G(W)$.
\end{list}
}

\noindent One can check that $u, C_1$ and $C_2$ are uniquely determined for any cartwheel $W$. We say $u$ is the \emph{hub} of $W$.

Let $W$ be a cartwheel and let $u, w$ be adjacent vertices of $G(W)$, one of which is the hub of $W$. We say that a rule $L$ \emph{appears with image $H$, source $u$ and sink $w$ in $W$} if there is a (planar) homeomorphism from $G(L)$ to $H$ taking $s$ to $u$ and $t$ to $w$ such that 
$\gm_L^-(v) \leq \gm_W(v') \leq \gm_L^+(v)$ for every $v \in V(G(L))$, where $v'$ is the image of $v$ under the homeomorphism. We then define the charge $o_{\mcal{L}}(u,w)$ carried from $u$ to $w$ by a set of rules $\mcal{L}$ analogously as $r_\mcal{L}(u,w)$ is defined for triangulations. The \emph{charge $c_{\mcal{L}}(W)$ of $W$ with respect to $\mcal{L}$} is 
$$
c_{\mcal{L}}(W) = 6 - d(u) - \sum_{w}o_\mcal{L}(u,w) + \sum_wo_\mcal{L}(w,u)
$$
where $u$ is the hub of $W$ and the sums are taken over all vertices adjacent to $u$.

If $T$ is a triangulation and $v \in V(T)$ a cartwheel $W$ \emph{appears in $T$ with hub $v$} if there is a spherical homeomorphism taking $G(W)$ to the second neighborhood of $v$, such that the hub of $W$ is mapped to $v$ and such that each vertex $w \in V(G(W))$ is mapped to a vertex with degree $\gm_W(w)$ in $T$. It is a short step from Lemma \ref{nn} to see that for every vertex $v$ of an internally 6-connected triangulation $T$ there is a cartwheel $W = W(v)$ appearing\footnote{This is where the proof necessitates spherical homeomorphisms; this statement would be false if cartwheel appearance were defined using planar homeomorphisms.} in $T$ with hub $v$. The cartwheel $W$ is unique up to homeomorphism and we speak of ``the'' cartwheel appearing with hub $v$. 

A configuration $K$ \emph{appears} in a cartwheel $W$ (or that $W$ \emph{contains} $K$) if there is a planar homeomorphism from $G(K)$ to an induced subgraph of $G(W)$ such that $\gm_L(v) = \gm_W(v')$ for every vertex $v \in V(G(K))$, where $v'$ is the image of $v$ in $W$.

The following lemmas relate what happens in a cartwheel (charge-wise and configuration-wise) to what happens in an internally 6-connected triangulation containing the cartwheel. The omitted proofs are straightforward.

\begin{lemma}\label{jan}
Let $T$ be an internally 6-connected triangulation and $\mcal{L}$ a set of rules. Then $c_\mcal{L}(v) = c_\mcal{L}(W(v))$ for every vertex $v$ of $T$.
\end{lemma}

\begin{lemma}\label{jon}
If a configuration $K$ appears in a cartwheel $W$ and $W$ appears in a triangulation $T$ then $K$ appears in $T$.
\end{lemma}

\noindent By Lemma \ref{jan} and \eqref{sumcharge} we have
\eqn
\label{positive}
\sum_{v\in V(T)} c_\mcal{L}(W(v)) = 12
\uneq 
for any internally 6-connected triangulation $T$ and any set of rules $\mcal{L}$. Thus 
the following lemma will prove the four-color theorem:

\begin{lemma}\label{appears}
An element of $\mcal{U}_{\ncmath}$ appears in every cartwheel $W$ such that $c_{\mcal{L}_{42}}(W) > 0$.
\end{lemma}

\noindent Indeed, any internally 6-connected triangulation $T$ must have a vertex $v$ such that $c_{\mcal{L}_{42}}(W(v)) > 0$ by \eqref{positive} so a configuration from $\mcal{U}_{\ncmath}$ appears in $W(v)$ by Lemma \ref{appears} and in $T$ by Lemma \ref{jon}. Then $T$ cannot be a minimal counterexample by Lemma \ref{cannot}. But every minimal counterexample 
is an internally 6-connected triangulation by Lemma \ref{int6} so minimal counterexamples do not exist, QED. 

As explained, we will divide the proof of Lemma \ref{appears} according to the degree of the hub of $W$. We first describe the proof for the case when $W$ has a hub of degree $\leq 11$. In this case the proof is given by machine-readable scripts called ``presentation files'' by Robertson et al. There is one presentation file for each hub degree from 5 to 11. Here we can only describe the format of a presentation file;
``doing'' the proof means running the computation. For cartwheels of hub degree 5 and 6 one can also prove Lemma \ref{appears} by hand, as do Robertson et al., but we will not take the time to do this. (However readers should easily be able to convince themselves from Fig$.$ \ref{mayerrules}, given that the configurations of Fig$.$ \ref{minorconfs} are in $\mcal{U}_{\ncmath}$.) While describing the files we refer to a generic unavoidable set $\mcal{U}$ and to a generic set of rules $\mcal{L}$, these two sets being loaded from files at the beginning of the computation.

The upcoming material explaining presentation files is more formally covered in Robertson et al.'s original article and in the notes accompanying their program \cite{rsst, discharge}. Here we just explain enough to point out some aspects of their implementation that affected our proof in particular ways. For example, certain choices of Robertson et al.~limited the type of configuration we could use, and also put restrictions on the set of rules. For reasons of simplicity our terminology does not always match the definitions of Robertson et al. As such our account is just an overview of ideas, not a faithful description of the discharging program.

We give names to the different parts of a cartwheel. Let $W$ be a cartwheel with hub $u$. The \emph{spokes} of $W$ are the vertices adjacent to $u$. The \emph{hats} of $W$ are the vertices at distance 2 from $u$ that are adjacent to exactly two spokes. The \emph{fan vertices} of $W$ are the vertices at distance 2 from $u$ that are adjacent to exactly one spoke. It is easy to check that every vertex is either the hub, a spoke, a hat or a fan vertex. The fan vertices adjacent to a spoke $v$ are called the \emph{fan vertices over~$v$}.

Let $W$ be a cartwheel. 
We say that an induced subgraph $H$ of $G(W)$ is a \emph{part graph of $W$} if the hub as well as every spoke and hat of $W$ is in $H$ and if all fan vertices of a spoke $v$ of $W$ are in $H$ when one fan vertex of $v$ is in $H$. 
We say that a graph $H$ is a \emph{part graph} if it is a part graph of some cartwheel $W$. 
The \emph{hub}, \emph{spokes}, \emph{hats} and \emph{fan vertices} of a part graph are defined in the natural way (as for a cartwheel, these are uniquely defined by the part graph).

A \emph{part} $P$ consists of a part graph $G(P)$ and two functions $\gm_P^-, \gm_P^+: V(G(P)) \rightarrow \{5, 6, 7, \ldots \} \cup \{ \infty \}$ such that (i) $\gm_P^-(v) = \gm_P^+(v) = d(v)$ for every internal vertex $v$ of $P$ where $d(v)$ is the degree of $v$ in $P$, (ii) every spoke $v$ with $\gm_P^-(v) = \gm_P^+(v)$ is an internal vertex of $G(P)$, and (iii) $\infty > \gm_P^-(v) \leq \gm_P^+(v)$ for every $v \in V(G(P))$. A cartwheel $W$ \emph{fits a part $P$} if there is a homeomorphic copy of $G(P)$ in $G(W)$ such that 
$\gm_P^-(v) \leq \gm_W(v') \leq \gm_P^+(v)$ for every $v \in V(G(P))$ where $v'$ is the image of $v$ in $G(W)$. Note the image of $G(P)$ may not be an induced subgraph of $G(W)$, as two hats may be adjacent in $G(W)$ but not in $G(P)$.

The \emph{trivial part of degree $d$} is the part $P$ with no fan vertices whose hub has degree $d$ and such that $\gm_P^-(v) = 5$, $\gm_P^+(v)= \infty$ for every spoke and hat $v$. We say that a part $P$ is \emph{dispatched} if every cartwheel $W$ that fits $P$ such that $c_{\mcal{L}}(W) > 0$ contains a configuration from $\mcal{U}$. Our goal is thus to dispatch the trivial part of degree $d$ for $d = 5, 6, \dots, 11$ since every cartwheel of hub degree $d$ fits the trivial part of degree $d$. The strategy is, given a part $P$ that we do not know how to dispatch, to refine the part into two complementary parts $P'$ and $P''$ such that every cartwheel $W$ that fits $P$ either fits $P'$ or $P''$, and to hope that $P'$ and $P''$ contain enough information to be dispatched by some direct observation---otherwise one keeps on refining the parts. 

\newcommand{\tone}{\tau_{\tt R}}
\newcommand{\ttwo}{\tau_{\tt H}}
\newcommand{\tthree}{\tau_{\tt S}}

Robertson et al.~use three tests $\tone$, $\ttwo$ and $\tthree$ to directly dispatch a part $P$. If the test $\tone(P)$ returns true then every cartwheel $W$ that fits $P$ contains an element of $\mcal{U}$, so $P$ is dispatched. If the test $\ttwo(P)$ returns true then every cartwheel $W$ that fits $P$ has $c_{\mcal{L}}(W) \leq 0$ unless $W$ contains an element of $\mcal{U}$, so $P$ is also dispatched. Finally the test $\tthree(P)$ returns true if a part $P$ is dispatched by reason of symmetry with previously dispatched parts. 
To dispatch the trivial part of degree $d$ one starts by placing the trivial part onto an empty stack. 
Then the following steps are repeated until the stack is empty: pop the last part $P$ on the stack if either $\tone(P)$ or $\ttwo(P)$ or $\tthree(P)$, otherwise choose parts $P'$, $P''$ 
such that every cartwheel $W$ that fits $P$ fits either $P'$ or $P''$, replace $P$ with $P'$, add $P''$ onto the top of the stack, and return to the first step. Clearly if the process terminates then the trivial part of degree $d$ is dispatched. The presentation file instructs the program whether to test for $\tone$, $\ttwo$ or $\tthree$ or else which parts $P'$, $P''$ to choose. Thus every line in a presentation file is one of four types: an order to test for either $\tone$, $\ttwo$ or $\tthree$, or a ``branching order'' specifying two complementary parts $P'$, $P''$ that refine $P$.

We give a brief overview of the tests $\tone$ and $\ttwo$ with special attention to implementation details that place restrictions on $\mcal{U}$ and $\mcal{L}$. Say that a part $P$ is \emph{reducible} if a configuration from $\mcal{U}$ appears in every cartwheel fitting $P$. Thus if $\tone(P)$ returns true $P$ should be reducible. Say that a configuration $K$ \emph{appears} in a part $P$ if there is an induced subgraph $H$ of $G(P)$ homeomorphic to $G(K)$ such that $\gm_P^-(v') = \gm_K(v) = \gm_P^+(v')$ for all vertices $v \in V(K)$, where $v'$ is the image of $v$ in $H$. Note that $K$ may appear in $P$ without appearing in every cartwheel $W$ fitting $P$. Indeed, while $H$ is an induced subgraph of $G(P)$, $G(P)$ is not necessarily an induced subgraph of $G(W)$, so the image of $H$ is not necessarily induced in $G(W)$ (recall that two hats may be adjacent in $G(W)$ but not in $G(P)$). 

\wellpositionedfig

To remedy this Robertson et al.~say a subgraph $H$ of $G(P)$ is \emph{well-positioned} if 
$H$ contains a spoke $v$ whenever it contains the two hats adjacent to $v$ in $P$. It is easy to see that if a configuration $K$ appears well-positioned in $P$ it appears in every cartwheel fitting $P$. This relationship is not if-and-only-if: a configuration could appear in every cartwheel fitting a part, but not appear well-positioned in the part itself, or indeed in any part (Fig$.$ \ref{wellpositioned}). The test $\tone$ only recognizes well-positioned configurations. While this may seem like a drawback configurations such as the one of Fig$.$~\ref{wellpositioned} are quite few. An advantage of this definition of appearance, used in the program, is that modifying the values of $\gamma_P^-(v)$ or $\gamma_P^+(v)$ for a vertex $v$ not in the image of the configuration does not affect whether the configuration appears well-positioned or not. 

\spanningtreefig
\bigradiusesfig

To check whether a given configuration $K \in \mcal{U}$ appears in a part $P$ a ``spanning subtree of faces'' is first prepared for $K$ by removing edges from $G(K)$ until no internal vertices are left (Fig$.$ \ref{spanningtree}); then the program tries to ``unfold'' the spanning subtree onto $G(P)$ triangle by triangle, choosing an initial edge and an orientation; 
the unfolding fails if a vertex is missing or if vertex degrees are incompatible. If the unfolding completes successfully it is not actually clear that $K$ appears in $P$---the mapping from $G(K)$ to $G(P)$ constructed by the unfolding may not be injective, or its image may not be induced, and so forth. 
Robertson et al.~have a lemma showing that $K$ appears in $P$ if the unfolding completes successfully as long as: 
(i) $K$ has radius $\leq 2$, 
where the ``radius'' of a configuration is the smallest $r$ such that there is a vertex in $G(K)$ at distance $\leq r$ from every other vertex in $G(K)$, 
(ii) $P$ has hub of degree $\geq 6$. The program
checks (i) at the start of the program when the set $\mcal{U}$ is loaded. If $\mcal{U}$ contains a configuration of radius $> 2$ the program terminates with an error message.
The latter behavior forced us to exclude 
certain $D$-reducible configurations that could have been interesting to use (Fig$.$ \ref{bigradiuses}). These are configurations which, despite having radius greater than 2, can fit inside a 
cartwheel by having their image exclude the hub. 
There are relatively few such configurations, though, and their loss does not seem to significantly impact the size of the unavoidable set. 

The discharging program also contains a safety check function which verifies ``rather crudely, from first principles'' that the mapping produced by the unfolding process is one-to-one, induced and well-positioned. 
Thus the correctness of the computation does not depend on the lemma; the lemma simply shows that, when conditions (i), (ii) are met and well-positionedness is checked, the mappings produced by the unfolding routine are 
\emph{theoretically} correct, and should never result in a complaint from the safety check function. The presence of the safety check function means that it is valid to use the program even if the conditions of the lemma aren't met. For example, one could remove the radius check by commenting out a line and use a set $\mcal{U}$ containing configurations of radius $> 2$; a successful computation would still indicate a correct discharging procedure. (We didn't do so because configurations of radius $> 2$ are not all that valuable, and we preferred not to modify the program in any qualitative way.) It is also correct to use the program with cartwheels of hub degree 5, eschewing condition (ii) of the lemma. 
(Despite the lemma, the program only checks that the hub degree is between 5 and 11.)

The test $\ttwo(P)$ should return true only if every cartwheel $W$ fitting $P$ contains a reducible configuration or else has $c_{\mcal{L}}(W) \leq 0$. 
The test depends on being able to upper bound the maximum net amount of charge that can be transferred from two spokes $u, v$ of a part $P$ to its hub in any cartwheel $W$ fitting $P$ that does not contain a configuration $K \in \m{U}$. We note this maximum net amount of charge as $\zeta(P, u, v)$. 
For every call to $\ttwo$ a sequence of triplets $S = ((u_1, v_1, q_1), \dots, (u_k, v_k, q_k))$ is specified in the presentation file, where $u_i, v_i$ are spokes of $P$ and $q_i \in \qq$ such that every spoke appears in exactly two triplets. 
The test $\ttwo$ 
returns true if $\zeta(P,u_i,v_i) \leq q_i$ for $i = 1, \dots, k$ and if $6-d + \lfloor (q_1 + \dots + q_k)/2 \rfloor_{10}\leq 0$, where $d$ is the hub degree and $\lfloor \cdot \rfloor_{10}$ denotes rounding down to the nearest fraction\footnote{Naturally, the program multiplies all charge values by 10 to make quantities integral; then one rounds down to the nearest integer.} of $\frac{1}{10}$ (i.e$.$ $\lfloor x \rfloor_{10} = \lfloor 10x \rfloor / 10$). The sequence $S$ is called a \emph{hubcap} for $P$. Possibly (and in fact quite often) a hubcap contains triplets where $u_i = v_i$. If $u = v$ then $\zeta(P,u,v)$ represents the maximum net charge transfer frum $u$ to the hub in any cartwheel fitting $P$, and $u$ still appears in exactly two triplets.

\unencodablefig

One upper bounds $\zeta(P,u,v)$ by enumerating 
all the ways incoming rules with source $u$ or $v$ can be placed onto the part. A combination is discarded if it triggers $\tone$; if not, the net charge transfer is computed by subtracting the charge of rules induced outward with sink $u$ or $v$ from the charge of rules induced inward with source $u$ or $v$. 
Doing this mainly requires a way of ``gluing'' rules onto a part to produce a new part, in order to check whether $\tone$ is triggered. The elegant method used by Robertson et al.~is to encode rules themselves as parts\footnote{In the program, rules are actually encoded as a more compact data structure called an \emph{outlet}. Outlets are efficient for representing parts that have few vertices $v$ with $(\gm_P^-(v), \gm^+(v)) \ne (5, \infty)$.} and to have a procedure for superposing or ``anding'' two parts together, which is easy to do. 
 Each rule becomes two parts, one in which the sink becomes 
the hub and one in which the source 
becomes the hub (if the degrees of the sink or source are not compatible with the hub degree under consideration, that part is not created). 
However not all rules can be encoded as parts, given that fan vertices do not exist in a part until the degree of the spoke beneath them has been fixed. For example the rule of Fig$.$~\ref{unencodable} cannot be encoded as a part where its sink coincides with the hub of the part, because the ``spoke'' that would be attached to the vertex of degree 5 has undetermined degree. Thus such a rule cannot be used with the program of Robertson et al. Once again, however, we did not find this to be particularly bothersome. 

We do not discuss the test $\tthree$, which dispatches a part by reason of symmetry with a previously dispatched part, as the reader should find such a test fairly intuitive. Fig$.$ \ref{presentation} illustrates a sample presentation file. To understand the figure some syntax is needed. Each line of the file starts with an `{\tt L}' followed by the recursion depth, which is the number of parts on the stack at that point. Lines that specify a branching order, called \emph{condition lines}, start with a `{\tt C}'. Lines that indicate a call to $\tone, \ttwo$ or $\tthree$ are called \emph{disposition lines} and start with {\tt R}, {\tt H} or {\tt S} respectively ({\tt H} is for ``hubcap'').
A condition line specifies a vertex $v$ of the last part $P$ on the stack such that $\gm_P^-(v) < \gm_P^+(v)$ and tells the program to branch either by raising $\gm_P^-(v)$ to a value $b \leq \gm_P^+(v)$ or by lowering $\gm_P^+(v)$ to a value $b \geq \gm_P^-(v)$. Each order has a \emph{complementary order}: in the first case, to lower $\gm_P^+(v)$ to $b - 1$, in the second case, to raise $\gm_P^-(v)$ to $b + 1$. The part $P'$ created by the original order is placed on top of the stack, whereas the part $P''$ created by the complementary order replaces $P$. (If $v$ is a spoke, an order which sets $\gm_P^+(v) = \gm_P^-(v)$ entails the creation of fan vertices for $v$, which are set to have lower degree 5 and upper degree $\infty$.) 

The following numbering scheme is used for vertices: numbers $1, \dots, d$ are the spokes in clockwise order, 
numbers $d + 1, \dots, 2d$ are the hats in clockwise order such that hat 1 is adjacent to spoke 1 and spoke 2, and finally number $k + (l+1)d$ where $1 \leq k \leq d$ and $l \geq 1$ is the $l$-th fan vertex clockwise above the $k$-th spoke, if the fan vertex exists. 
Altogether the format of a condition line is ``{\tt C} $m$ $n$'' where $m$ is the number of the vertex $v$ to be affected and $n \in \{\dots, -6, -5, 6, 7, \dots\}$ specifies the modification to the degree function. If $n > 0$ the order is to increase $\gm_P^-(v)$ to $n$, otherwise the order is to decrease $\gm_P^+(v)$ to $-n$.  

\presentationfig

The triplets of numbers in an `{\tt H}' line list the hubcap. Identical triplets are only listed once and charge values are multiplied by 10 in order to avoid fractions. 
When reading a presentation file the hubcap lines are by far the most difficult to check, so difficult that it is debatable to what degree they are really ``readable''. 
One can muster some help with the `verbose' output option of the program,
which gives information as to which reducible configurations are appearing along with other helpful details. While the resulting output may be readable at a normal pace it is also quite large: over 3'000'000 lines for the Robertson et al.~proof, over 13'000'000 lines for our proof. A mathematician checking these proofs at the rate of one line per second and working 9 hours a day would take over 3 months to read the Robertson et al.~proof and over a year to read ours. 
We refer readers with further interest in presentation files to \cite{discharge}.

The proof of Lemma \ref{appears} requires a separate argument to show that vertices of degree $\geq 12$ are not overcharged. For this it suffices to show that $r_{\m{L}_{42}}(u,v) \leq \frac{1}{2}$ for any two adjacent vertices $u$, $v$ of an internally 6-connected triangulation $T$ where $v$ has degree $\geq 12$ as long as $T$ does not contain an element of $\m{U}_{\ncmath}$, since the initial negative charge of vertices of degree $\geq 12$ is at least half their degree in absolute value. This is proved in Lemma \ref{lastly} below. The configurations that are used for the proof appear in Fig$.$ \ref{proofconfs} (these configurations all appear on the first page of  ${\tt U\_\ncmath.pdf}$---which is ordered identically to  ${\tt U\_\ncmath.conf}$---making it easy to check they are in $\mcal{U}_{\ncmath}$).
The proof refers to rules by the order which they appear in Fig$.$ \ref{mainrules}, from left to right and top to bottom. Serious readers are encouraged to pencil in some numbers.

\newcommand{\FtFfSoFfFtFfSoFf}{3}
\newcommand{\FtFfSoFsFtSfSoFf}{4}
\newcommand{\FtFfSoSsFtSsSoFf}{5}
\newcommand{\FtFsSoFsFtSfSoSf}{6}
\newcommand{\FtFsSoSfFtFsSoSf}{7}
\newcommand{\FtFsSoSsFtSsSoSf}{8}

\proofconfsfig

\begin{lemma}\label{lastly}
Let $u, v$ be adjacent vertices of an internally 6-connected triangulation $T$, where $v$ has degree at least 8. Then $r_{\m{L}_{42}}(u, v) \leq \frac{1}{2}$ if no members of $\m{U}_{\ncmath}$ appear in $T$.
\end{lemma}
\begin{proof}
We say that a rule \emph{appears} in $T$ to mean it appears with source $u$ and sink $v$. If $u$ has degree $\geq 9$ then $r_{\m{L}_{42}}(u,v) = 0$ because there are no rules with source of degree $\geq 9$ in $\m{L}_{42}$. Thus we can assume that $u$ has either degree 5, 6, 7 or 8, and we analyze each of these cases separately. Except for the first case rules all have value $\frac{1}{10}$ and the object is to show that no more than 5 rules (multiplicities counted) can appear simultaneously, without having a configuration in $\m{U}_{\ncmath}$ appear. 

We proceed with the four cases.\\

\noindent \emph{Vertex $u$ has degree 5.} In this case only the first four rules of $\m{L}_{42}$ can appear. Each can appear at most twice except for rule 1 which can (and will) appear only once. Say rule 4 appears twice. Then rules 2 and 3 cannot appear at all because of configurations 1 and 2 (Fig$.$ ), so $r_{\m{L}_{42}}(u, v) = \frac{2}{5} < \frac{1}{2}$. 
Otherwise say rule 3 appears twice. Then rule 2 cannot appear because of configuration 1 and rule 4 cannot appear because of configurations 1 and 2 again, so $r_{\m{L}_{42}}(u, v) = \frac{2}{5}$ again. Likewise if rule 2 appears twice then rule 4 cannot appear and rule 3 can appear only once, so $r_{\m{L}_{42}}(u, v) \leq \frac{2}{5} + \frac{1}{10} = \frac{1}{2}$. Therefore we can assume rules 2, 3 and 4 appear at most once each, for a possible total of $\frac{1}{5} + \frac{3}{10} = \frac{1}{2}$.\\

\degreetwelvefig

\noindent \emph{Vertex $u$ has degree 6.} In this case the only rules which can appear are rules 2 through 7. Say rule 7 appears twice. Then rules 2, 3 and 6 cannot appear, and rules 4, 5 can only appear twice between the two of them because of configuration 2. Thus the charge transferred is at most $\frac{1}{10} + \frac{1}{10} + \frac{2}{10} < \frac{1}{2}$.
Now say rules 2 and 7 both appear. If rules 3, 4, 5 total 4 appearances between them then rule 4 must appear twice and configuration 2 appears, so rules 3, 4, 5 can only total 3 appearances between them, which makes a possible total of 5 rule appearances. If rule 7 appears and rule 2 does not appear then rule 6 does not appear and rules 4, 5 cannot both appear twice because of configuration 2, so the total number of rule appearances is at most 5, since rule 3 can only appear once. Therefore we can assume rule 7 does not appear.

If rule 2 appears twice then rule 4 cannot appear twice because of configurations \FtFfSoFfFtFfSoFf, \FtFfSoFsFtSfSoFf, \FtFsSoFsFtSfSoSf{} and rules 3, 6 can only appear twice between the two of them so there are most 5 rule appearances. If rule 2 appears once then rules 3, 5, 6 can each appear only once and if they each appear once then rule 4 cannot appear twice because of configuration \FtFsSoSfFtFsSoSf, so there are at most 5 configuration appearances. On the other hand if rule 2 does not appear then rule 6 does not appear and rules 3, 4, 5 cannot each appear twice because of configuration 2, so total at most 5 rules appearances between them.\\

\noindent \emph{Vertex $u$ has degree 7.} This is the lengthiest case, as there are the most rules with source of degree 7 and sink of unbounded degree. To ease the analysis we have shown the cases in Fig$.$ \ref{degreetwelve}. 

If $u$ has degree 7 then one of the rules in the left column of Fig$.$ \ref{degreetwelve} appears with source $u$ and sink $v$ in $T$, as is easy to check (these rules are not part of $\m{L}_{42}$, they are just a device for organizing the case analysis).
Next to each of these rules we list the numbers of the rules in $\m{L}_{42}$ that can appear in $T$ when the rule appears; a number is listed twice if a rule can appear twice, in which case one should think of the first occurence of the number as representing the ``right side up'' appearance of the rule as drawn in Fig$.$ \ref{mainrules}, and the second occurence of the number as representing ``upside down'' occurence of the rule obtained by a horizontal mirror symmetry. If a rule cannot appear without creating the appearance of a configuration from Fig$.$ \ref{proofconfs} its number is crossed out. A set of rules that cannot appear simultaneously or that cannot appear simultaneously without creating a configuration from Fig$.$ \ref{proofconfs} is underlined. A set of rules that are pairwise incompatible is overlined. 

In all cases it is easy to check that at most 5 rules can appear simultaneously. For example, in the next-to-last case, only three rules from the set $\{8, 8, 9, 9\}$ can appear because that group is underlined and only one rule from the overlined group can appear; this makes a maximum of $3 + 1 + 1 = 5$ rule appearances.\\ 

\noindent \emph{Vertex $u$ has degree 8.} There are only 4 rules in $\m{L}_{42}$ with hub degree 8 and unbounded sink degree, rules 39 through 42. Since only rule 43 is compatible with its mirror image (rule 41 has no mirror image) only 5 rules can appear among these.

\end{proof}

\noindent \textbf{Acknowledgments.} I would like to thank Phillip Rogaway, my Ph.D. advisor, for allowing me to work on this project outside the scope of his research. My greatest thanks to Tatsuaki Okamoto and the team at NTT labs for hosting me and my work for two months. Thanks also to Walter Stromquist and Wolfgang Haken who helped me with references, to Robin Thomas for data on $D$-reducibility, and to Zacharia Johnson for helping me set up a distributed computation at UC Davis. Finally I would like to give my full thanks to Izabella Laba, my sponsor at the University of British Columbia.

\end{document}